\documentclass{amsart}

\usepackage[utf8]{inputenc}
\usepackage[T1]{fontenc}

\usepackage[usenames, dvipsnames]{color}
\usepackage{graphicx}
\usepackage{amssymb}

\newcommand{\C}{\mathbb C}
\newcommand{\B}{\mathbb B}
\newcommand{\D}{\mathbb D}
\newcommand{\Q}{\mathbb Q}
\newcommand{\R}{\mathbb R}
\newcommand{\N}{\mathbb N}

\newcommand{\Z}{\mathbb Z}
\newcommand{\diff}{\mathop{}\!\mathrm{d}}
\newcommand{\pardiff}[2]{\frac{\partial #1}{\partial #2}}
\newcommand{\norm}[1]{\left\Vert #1 \right\Vert}
\newcommand{\abs}[1]{\left| #1 \right|}

\DeclareMathOperator{\intr}{int}
\DeclareMathOperator{\Gr}{Gr}

\newtheorem{proposition}{Proposition}
\newtheorem{lemma}{Lemma}
\newtheorem{theorem}{Theorem}
\newtheorem{corollary}{Corollary}
\newtheorem{conjecture}{Conjecture}
\theoremstyle{definition}
\newtheorem{remark}{Remark}

\begin{document}

\title[Approximation by Random Polynomials and Rational functions]{Approximation by Random Complex Polynomials and Rational functions}

\author{S. St-Amant and J. Turcotte}

\address{Département de mathématiques et de statistique, Université de Montréal,
CP-6128 Centreville, Montréal,  H3C3J7, CANADA}
\email{simon.st-amant@umontreal.ca,jeremie.turcotte@umontreal.ca}

\color{black}

\begin{abstract}
We seek random versions of some classical theorems on complex approximation by polynomials and rational functions, as well as investigate properties of random compact sets in connection to complex approximation.
\end{abstract}

\maketitle

\section{Introduction}

The main objective was to generalize complex approximation theorems to the case of random functions. Our aim was to find results similar to those of Andrus and Brown in \cite{AB1984} and Istrăţescu and Onicescu in \cite{OI1978}, while attempting to be as general as possible but at the same time avoiding "almost everywhere" results as much as possible.

We have obtained results comparable to those in the preceding articles. Our main results are the following :
\begin{itemize}
    \item A generalization of Runge's theorem (Theorem \ref{Runge});
    \item A generalization of the Oka-Weil theorem (Theorem \ref{OkaWeil});
    \item The evaluation of a random function over a compact set is a random compact set (Theorem \ref{F sub f});
    \item The polynomially and rationally convex hulls of a random compact set are random compact sets (Theorem \ref{P-hull random} et \ref{R-hull random});
    \item The Siciak extremal function and the pluricomplex Green function of a random compact set are random functions (Theorem \ref{siciak});
    \item A useful convergence theorem, that states that from a weak form of convergence we may extract stronger convergence (Theorem \ref{convUnif}).
\end{itemize}

There are  two natural approaches to stochastic complex approximation. The first is to follow the proof of the deterministic case and prove that each step preserves measurability. This approach  is used to prove the first result in our list.

The second approach is to apply measurable selection theorems, as the classical theorems always tell us it is possible to approximate our functions, but not that this approximation is measurable. We wish  to underline here the importance of our convergence theorem. It shows (in the appropriate context) that pointwise convergence implies  joint (uniform) convergence. Therefore, the real difficulty in generalizing theorems is found in the measurability part of the proof, since only a rather weak convergence is necessary. It follows that if we had stronger measurable selection theorems, we would be able to prove stronger random versions of classical complex approximation theorems.

We have only scarcely looked at measure spaces, instead concentrating on measurable spaces. The former would be very interesting, but our approach guarantees generality and thus all results also apply to measure spaces (as they are measurable spaces).

This research was supported by a grant from NSERC (Canada) under the supervision of Paul M. Gauthier. We also thank Eduardo Zeron for helpful conversations.


\section{Random functions}
By a measurable space $(Y,\mathcal A)$ we mean a set $Y$ with a 
$\sigma$-algebra 
$\mathcal A$ of subsets of $Y$ and the members of $\mathcal A$ are called the measurable sets of $Y.$ If $Y$ is a 
topological space,
we choose as $\sigma$-algebra  the family $\mathcal B$ of Borel subsets of $Y$ and consider the measurable space $(Y,\mathcal B).$ A function $f:Y\rightarrow Z$ from a measurable space  $(Y,\mathcal A)$ to a measurable space $(Z,\mathcal C)$ is said to be measurable if $f^{-1}(E)$ is measurable, for every measurable set $E$ in $Z.$  

As in \cite{AB1984}, if $(\Omega, \mathcal A)$ and $(Z, \mathcal B)$ are two measurable spaces and $X$ is an arbitrary non-empty set, then a function $f:\Omega\times X\rightarrow Z$ is  said to be a {\em random function} if the function $f(\cdot,x)$ is measurable for each $x\in X.$ Clearly, if  $f:\Omega\times X\rightarrow Z$ is a random function and $Y\subset X,$ then the restriction mapping  $f:\Omega\times Y\rightarrow Z$ is a random function. 

We remark that a random measurable function need not be measurable. That is, a function $f(\omega,x),$ such that $f(\cdot,x)$ and $f(\omega,\cdot)$ are both measurable, need not be measurable. To put it differently, a function which is separately measurable need not be jointly measurable. An example is given by Sierpinsiki (see \cite[p. 167]{Ru1987}). However, Neubrunn \cite{Ne} has shown the following result. Let $X$ be a separable metric space. Let $Y$ be an abstract set and $T$ a $\sigma$-algebra of subsets of $Y.$ Let $f(x,y)$ be a real function defined on $X\times Y$ such that for every fixed $y\in Y$ the function $f^y(x)=f(x,y)$ is continuous on $X$ and for fixed $x\in X$ the function $f_x(y)=f(x,y)$ is measurable in $(Y,T)$ 
 Then, the function $f(x,y)$ is measurable in the space $(X\times Y,S\times T).$ Thus, a function which is measurable in one variable and continuous in the other 
(Such functions are called Carathéodory functions.) is measurable.  
The same result holds for complex-valued functions.


\section{Random Holomorphic Functions; one variable}

If $U$ is an open set in $\C^n,$ we say that a function $F:\Omega\times U\rightarrow \C$ is a random holomorphic function on $U$ if $F(\omega,\cdot)$ is holomorphic for each $z\in U$ and $F(\cdot,z)$ is measurable for each $z\in U.$ 

If $E\subset \C^n,$ we say that a function $f:\Omega\times E\rightarrow \C$ is a random holomorphic function on $E$, if 
for every $\omega \in \Omega$,
there is an open neighbourhood $U$ of $E$ and a random holomorphic  function $F$ on $U$ such that $F(\omega,\cdot)=f(\omega, \cdot)$ on $E$.

Let $K$ be a compact set in $\C^n$ and $f$ be a 
random continuous function on $K$ which is holomorphic on the interior of $K$. We say that $f$ is in $R_\Omega(K)$ if there exists a sequence $r_j$ of random rational functions 
with no poles on $K$ such that, 
for every $\omega \in \Omega,$
 $r_j(\omega, \cdot) \rightarrow f(\omega, \cdot)$ uniformly .
Similarly, we say that $f$ is in $R_\Omega^{unif}(K)$ if for every $\varepsilon > 0$ there exists a random rational function $r$ with no poles on $K$
such that $\abs{r(\omega, z) - f(\omega, z)} < \varepsilon$ for all $(\omega, z)$ in $\Omega \times K$.

We have the following rather strong result stating that separately uniform convergence implies joint uniform convergence.

\begin{theorem} \label{weakConvUnif}
    $R_\Omega^{unif}(K) = R_\Omega(K).$
\end{theorem}

\begin{proof}
    This is an application of Corollary \ref{unif} 
which can be found in Section \ref{GeneralizedRandomFunctions}. 
\end{proof}

It would thus be useful to have sufficient conditions  for a random function $f$ to be in  $R_\Omega(K).$ Let us call such a result a Runge theorem. 
We can formulate a first version of a Runge theorem.

\begin{theorem}\label{Runge 1}
Let $U$ be an open set in $\C$ and $f:\Omega\times U \rightarrow \C$ be a random holomorphic function. Let $K$ be a compact subset of $U$. 
Then, there exists a sequence $R_1, R_2, \ldots,$ of random rational functions with polls off 
$K,$ such that,  for each
$\omega\in\Omega,$ 
$$
	R_n(\omega,\cdot)\rightarrow f(\omega,\cdot) \quad \mbox{uniformly on} \quad K. 
$$
\end{theorem}

\begin{proof}
We can cover $K$ by a finitely many disjoint compact sets $Q_1,\cdots,Q_n,$ such that each $Q_k$ is contained in $U$ and each $Q_k$ is bounded by finitely many disjoint polygonal curves. 
Set $\Gamma=\cup_j\partial Q_k.$ By the Cauchy formula, for each $\omega\in\Omega,$ 
$$
	f(\omega, z) = \frac{1}{2\pi i} \int_\Gamma \frac{f(\omega, \zeta)}{\zeta - z}\diff\zeta		
		\quad,\quad  \forall z\in K.
$$
For $\delta>0,$ partition $\Gamma$ into $N=N(\delta)$ segments $\Gamma_j$ of length smaller than $\delta$.  For each $\Gamma_j,$ denote by $\zeta_j$ the terminal point of $\Gamma_j.$ 
The Riemann sum 
\begin{equation*}
R(\omega, z) = \sum_{j=1}^{N(\delta)} \frac{1}{2\pi i} \frac{f(\omega, \zeta_j)}{\zeta_j - z} \int_{\Gamma_j} \diff\zeta =  \sum_{j=1}^{N(\delta)} \frac{a_j(\omega)}{\zeta_j-z}
\end{equation*}
is a random rational function.
Put
$$
	 \eta(\omega,\delta) = \max 
\left\{\frac{1}{2\pi}\left|\frac{f(\omega, \zeta)}{\zeta - z} - \frac{f(\omega, w)}{w - z}\right|:
	\zeta, w \in \Gamma, |\zeta-w|<\delta,  z\in K\right\}. 
$$
For all $(\omega,z)\in\Omega\times K,$ 
$$
	\abs{f(\omega,z) - R(\omega,z)} < \eta(\omega,\delta)\cdot L(\Gamma),
$$
where $L(\Gamma)$ is the length of $\Gamma.$ It follows from the uniform continuity of $f(\omega,\zeta)/(\zeta-z)$ on $\Gamma\times K,$ that if $\delta=\delta(\omega)$ is sufficiently small, then $\eta(\omega,\delta)<\varepsilon/L(\Gamma).$ Thus, 
$$
	\abs{f(\omega,z) - R(\omega,z)} < \varepsilon, \quad \forall z\in K. 
$$

Let $\{\delta_n\}_n$ be a sequence of positive numbers decreasing to zero and, for each $\delta_n,$  let $R_n$ be a random rational function corresponding as above to $\delta_n.$ Then, for each
$\omega\in\Omega,$ 
$$
	R_n(\omega,\cdot)\rightarrow f(\omega,\cdot) \quad \mbox{uniformly on} \quad K. 
$$

\end{proof}

However, the condition that $f(\omega, \cdot)$ be holomorphic on 
$U$ for all $\omega \in \Omega$ is quite strong and we would like to find a better version of Runge's theorem such as the following.

\begin{conjecture} \label{conjRunge}
Let $f(\omega,z)$ be a random holomorphic function on a compact set $K\subset \C.$ Then, there is a sequence $R_k(\omega,z)$ of random rational functions such that, for each $\omega,$ $R_k(\omega,\cdot)\rightarrow f(\omega,\cdot)$ uniformly on $K$. 
\end{conjecture}

From Conjecture \ref{conjRunge} and Theorem \ref{weakConvUnif}, 
the following stronger version of Runge's theorem follows.

\begin{conjecture} \label{Runge}
    Let $f$ be a random holomorphic function on a compact set $K \subset \C$ such that, 
for all $\omega \in \Omega$
there exists an open set on which $f(\omega, \cdot)$ is holomorphic. Then, $f \in R_\Omega^{unif}(K)$.
\end{conjecture}


\section{Random Holomorphic Functions; several variables}

If  $U$ is an open subset of $\C^n,$ and $(\Omega, \mathcal A)$ is a measurable space, a random function $f:\Omega\times U\rightarrow \C$ is a random holomorphic function on $U,$ if $f(\omega,\cdot)$ is holomorphic, for each $\omega\in \Omega.$ As for random smooth functions (see Section \ref{annex}), for $\alpha \in \N^n$, we denote by $f^{(\alpha)}$ the function defined for every fixed $\omega \in \Omega$ as $f(\omega, \cdot)^{(\alpha)}$. However, the derivative is now meant complex. 

\begin{lemma}
If $f$ is a random holomorphic funtion, then, for each muli-index $\alpha,$ the function $f^{(\alpha)}$ is also a random holomorphic function.
\end{lemma}

\begin{proof}
The proof is essentially the same as for Lemma \ref{smoothlemma} since the holomorphy of $f$ ensures that the differential quotient tends to $f(\omega,z)^{(\alpha)}$ and the measurability of $f^{(\alpha)}$ is guaranteed by Proposition \ref{limsmooth} in the annex.
\end{proof}

A sequence $f_j, j\in \N$ of functions, defined on an open set $\Omega\subset \C^n,$
is said to converge compactly if it converges uniformly on compact subsets. A sequence $f_j, j\in \Z$ is said to converge compactly if the sequences $f_j, j=0,1,\ldots,$ and $f_j, j=-1,-2,\ldots$ converge compactly. 
If a sequence $f_j, j\in \Z,$ converges compactly, we set 
$$
	\lim_{j\rightarrow \infty}f_j = \lim_{j\rightarrow-\infty}f_j + \lim_{j\rightarrow+\infty}f_j.
$$
It is then clear how to define the compact convergence of series of functions $\sum_{j=0}^{+\infty} f_j$ and $\sum_{j=-\infty}^{+\infty}f_j.$

For $n\in\N,$ we say that a  multiple series of the form 
$$	
	\sum_{\alpha\in \Z^n}a_\alpha, \quad  a_\alpha\in \C 
$$
converges if it converges absolutely. If $f_\alpha, \alpha\in \Z^n,$ is an indexed family of functions in an open set $\Omega\subset\C^n,$ we shall say that the multiple 
series $\sum_\alpha f_\alpha$ converges compactly to a function $f$ in $\Omega$ if it is pointwise convergent to $f$ and  for every compact set $K\subset \Omega$ and for every $\varepsilon>0,$ there is a finite set of indices $\Lambda\subset\Z^n,$ such that 
$$
	|f(z)-\sum_{\alpha\not\in\Lambda} f_\alpha(z)|<\varepsilon, \quad  \forall z\in K.
$$  

We shall say that a formal Laurent series 
$$
	\sum_{\alpha\in \Z^n}a_\alpha z^\alpha, \quad  z\in \Omega
$$
converges normally in $\Omega$ if for every compact set $K\subset \Omega$ and for every $\varepsilon>0,$ there is a finite set of indices $\Lambda\subset\Z^n,$ such that 
$$
	\sum_{\alpha\not\in\Lambda} |a_\alpha z^\alpha|<\varepsilon, \quad  \forall z\in K.
$$  
If a formal Laurent series converges normally, then every arrangement converges compactly in $\Omega$ to a unique limit $f.$  It follows that the series converges compactly to $f.$ 
It also follows that the homogeneous series 
$$
	\sum_{m=-\infty}^{+\infty}P_m(z), \quad \mbox{where}  \quad 
		 P_m(z)=\sum_{|\alpha|=m}a_\alpha z^\alpha,
$$
converges compactly to $f$ in $\Omega.$ 
For the basic properties of such multiple series, see  \cite[Sections I.1.5 and II.1.2 ]{RA1986}.

An expression of the form 
$$	
	\sum_{\alpha}a_\alpha(\omega)z^\alpha,
$$
where $\alpha=(\alpha_1,\ldots,\alpha_n), \alpha_j\in \N\cup\{0\},$ is a random power series, if each $a_\alpha$ is measurable.

\begin{theorem}\label{Taylor}
    If $f$ is a random holomorphic function on the  polydisc $\D^n=\{z: |z_j|< 1, j=1,\dots, n\},$ then $f$ has a representation as a sum of a random power series
$$
	f(\omega,z) = \sum_{\alpha}a_\alpha(\omega)z^\alpha.
$$
\end{theorem}

\begin{proof}
The proof follows directly from the previous lemma, since
\begin{equation*}
    a_\alpha(\omega) = \frac{f(\omega,0)^{(\alpha)}}{\alpha!} \quad 
		\mbox{with} \quad \alpha! = \alpha_1\cdots\alpha_n
\end{equation*}
and the $a_\alpha$ are measurable by the measurability of $f(\cdot,0)^{(\alpha)}$.
\end{proof}

If $f$ is holomorphic in a polydisc 
$$
	D(a,r) = \{z\in\C^n:|z_j-a_j|<r_j, \, j=1,\ldots,n\},
$$ 
then the Taylor expansion of $f$ at $a$ converges to $f$ normaly on $D(a,r).$ Denote by 
$$
	P_m(z-a) = \sum_{|\alpha|=m}a_\alpha (z-a)^\alpha,
$$
the sum of all terms of degree $m$ in the Taylor expansion. Then, $f$ has the  {\em homogeneous expansion}
$$
	f(z) = \sum_{m=0}^\infty P_m(z-a), 
$$
in terms of the m-homogeneous polynomials $P_m$ and this expansion converges normaly in $D(a,r).$

If $f$ is holomorphic in the unit ball  $\B^n$ of $\C^n,$ then the Taylor expansion of $f$ about the origin converges normaly to $f$  on $\B^n.$  Indeed, for every compact subset $K\subset \B^n,$  there are finitely many open polydiscs $D_j, j=1,\ldots,k$ centered at $0,$ and contained in  $\B^n,$ such that $K\subset D_1\cup\cdots \cup D_k.$ Again, by compactness, there is a $0<\lambda<1,$ such that $K$ is contained in the compact set 
$$
	Q= \bigcup_{j=1}^k\lambda \overline D_j  = 
	\bigcup_{j=1}^k(K\cap\lambda \overline D_j)  =  				  	\bigcup_{j=1}^kK_j.
$$  
For each $j,$ the homogeneous expansion convrges normaly on $D_j,$ so the homogeneous expansion converges uniformly on $K_j.$ It follows that the homogeneous expansion converges uniformly on $Q$ and hence on $K.$ Combining this with Theorem \ref{Taylor}, and since partial sums of the homogeneous expansion are polynomials, we have the following result.

\begin{theorem}
Let $f(\omega,z)$ be a random holomorphic function on $U$, where $U$ is an open polydisc $|z_j|<r_j, j=1,\ldots,m,$ or an open ball $\|z\|<r.$ Then,  the partial sums $p_m(\omega,z), m=0,1,\ldots,$ of the random homogeneous expansion of $f$  consititute a sequence of random polynomials which, for every $\omega$ and  for every compact $K\subset U,$ converges uniformly on $K$ to $f(\omega,\cdot).$ 
\end{theorem}

This theorem follows directly from the previous discussion and the fact that a linear combination of measurable functions is measurable.

\begin{lemma}\label{homogeneous unique}
Suppose
$$
	\sum_{m=0}^\infty P_m(w) = \sum_{m=0}^\infty Q_m(w),
$$
for all $w$ in some neighbourhood of $0,$ where $P_m$ and $Q_m$ are $m$-homogeneous polynomials. 
Then, $P_m=Q_m, \, m=0,1,\, \ldots.$ 
Suppose, in particular, that $\sum P_m$ is the homogeneous expansion of a function $f$ holomorphic in a neighbourhood of $0.$ Let $z=R(w)$ be a linear change of coordinates and let $\sum Q_m$ be the homogeneous expansion of $f\circ R^{-1}.$ Then, $Q_m=P_m\circ R, \, m=0,1,\, \ldots.$ 
\end{lemma}

\begin{proof}
We may assume that the equality holds on the unit ball $\overline B.$ Fix some point $\theta\in\partial B.$ Then, for each complex number $\lambda, |\lambda|\le 1,$ 
$$	
	\sum_m P_m(\theta)\lambda^m=\sum_mP_m(\lambda\theta)=
	\sum_mQ_m(\lambda\theta)=\sum_m Q_m(\theta)\lambda^m.
$$
Therefore, $P_m(\theta)=\Q_m(\theta),$ for all $m.$ Since, $\theta\in\partial B$ was arbitrary, we have that $P_m=Q_m$ on $\partial B.$ By the maximum principle, $P_m=Q_m$ in $B$ and hence everywhere. 
 
\end{proof}

\begin{theorem}\label{lambda z}
Let $f(\omega,z)$ be a random holomorphic function in $U$, where $U$ is an open set with the property that if $z\in U$, then $\lambda z\in U$ for all $|\lambda|\leq 1$. Then, there is a sequence $p_m(\omega,z), m=0,1,\ldots,$ of random polynomials such that, for every $\omega,$  the sequence $p_m(\omega,\cdot)$ converges compactly on $U$ to $f(\omega,\cdot).$ 
\end{theorem}

\begin{proof} 
Let $U$ be an open subset of $\C^n$ such that, for each $a\in U,$ we also have $\lambda a\in U,$ for each $\lambda\in\C$ with $|\lambda|\le 1.$  
Let $f(\omega,w)$ be a  random holomorphic function in $U$ and fix a point
$a\in U\setminus \{0\}$ 
and $\varepsilon>0.$  Let $R_a$ be a unitary transformation such that $R_a(a)=(|a|,0,\ldots,0).$ 
There is a polydisc $D_a$ centered at $0$ and containing $(|a|,0,\ldots,0),$ such that $R_a^{-1}(D_a)\subset U.$  It is easy to see that $f(\omega,R_a^{-1}(z))$ is a random holomorphic function on $D_a$ and hence 
has a random homogeneous expansion
$$
	f(\omega,R_a^{-1}(z)) = \sum_{m=0}^{\infty}P_m(\omega,z), \quad z\in D_a,
$$
which, for each $\omega,$ converges uniformly on $D_a.$ Thus,
\begin{equation}\label{a expansion}
	f(\omega,w) = \sum_{m=0}^{\infty}P_m(\omega,R_a(w)), \quad w\in R_a^{-1}(D_a)
\end{equation}
converges, for each $\omega,$ uniformly on the open neighbourhood  $R_a^{-1}(D_a)$ of $a.$ Hence, there is an $M=M_{\omega,\varepsilon},$ such that for 
$$
	P_a(\omega,w) =  \sum_{m=0}^MP_m(\omega,R_a(w)), \quad w\in R_a^{-1}(D_a),
$$
we have  that $|f(\omega,w)-P_a(\omega,w)|<\varepsilon,$ for all $w\in  R_a^{-1}(D_a).$ 

Since homogeneous expansions are unique (Lemma \ref{homogeneous unique}), each of the expansions (\ref{a expansion}) is the random homogeneous expansion 
$$
	f(\omega,w) =  \sum_{j=0}^{\infty}Q_j(\omega,w)
$$
of $f(\omega,w)$ at $0.$ 

Let $p_m(\omega,w)$ be the $m$-th partial sum of this homogeneous expansion and
fix $\omega.$ We have shown that 
for each $a\in U,$ we have a neighbourhood $V_a$ and an $m=m_a$ such that  $|f(\omega,w)-p_m(\omega,w)|<\varepsilon,$ for all $w\in V_a.$ 
Therefore, if $K$ is a compact subset of $U,$ we can cover $K$ by finitely many such open sets $V_1,\ldots, V_k$ and conclude that there is an $m=m_K$ such that  $|f(\omega,w)-p_m(\omega,w)|<\varepsilon,$ for all $w\in K.$

\end{proof}

We recall that a domain (connected open set)  $U\subset\C^n$ is a connected Reinhardt domain if, for every $z=(z_1,\ldots,z_n)\in U$ all points $(e^{i\theta_1}z_1,\ldots,e^{i\theta_n}z_n)$, with  $\theta_1,\ldots,\theta_n\in\R,$ are also in $U.$ If $U$ satisfies the hypotheses of the previous theorem, then $U$ is a Reinhhardt domain. 
\color{black}
We now generalize the previous theorem to Laurent series. For an introduction to Laurent series, see \cite{K, PW, RA1986}.

\begin{theorem}\label{Laurent}
Let $f(\omega,z)$ be a random holomorphic function in $U$, where $U$ is a connected Reinhardt domain. Then, f has a unique expansion as a random Laurent series and for every $\omega,$ the series converges normaly on $U$ to $f(\omega,\cdot).$ 
\end{theorem}

\begin{proof}
Without loss of generality, assume the Reinhardt domain is centered at $0$. For fixed $\omega$, by  \cite[Th 1.5, p. 46]{RA1986}, there exists a unique Laurent series
\begin{equation*}
\sum_{\nu \in \Z^n} c_\nu(\omega) z^\nu
\end{equation*}
converging normally to $f(\omega, \cdot)$. Thus the Laurent series converges absolutely and uniformly on the compact subsets of $U$. The coefficients $c_\nu$ are given by the integral formula
\begin{equation}\label{c nu}
c_\nu(\omega) = \frac{1}{(2\pi)^n} \int_{bP} \frac{f(\omega,\zeta)}{\zeta^{\nu + 1}} \diff\zeta_1 \dots \diff\zeta_n,
\end{equation}
where $bP$  is the distinguished boundary of a polydisc $P(0,\rho)$ centered at $0$ with multiradius $\rho$ chosen such that its closure is contained in  $U$. Also, $\zeta^{\nu+1}$ is shorthand for $\zeta_1^{\nu_1+1}\zeta_2^{\nu_2+1}\cdots\zeta_n^{\nu_n+1}.$ 
The measurability of $c_\nu$ follows from using Proposition \ref{randomInt} multiple times since the function $F(\omega, \zeta) = f(\omega,\zeta)/\zeta^{\nu+1}$ is continuous and therefore Riemann integrable for fixed $\omega$.

\end{proof}

\begin{corollary}
Let $f(\omega,z)$ be a random holomorphic function in a connected Reinhardt domain $U.$ Then, there is a sequence of random  rational functions $r_k(\omega,z)$ such that, for each $\omega$ the sequence $r_j(\omega,\cdot)$ converges compactly to $f(\omega,\cdot)$ on $U.$ 
\end{corollary}
      
The theorems in this section can be considered as special cases of Theorem \ref{Laurent}. For example, suppose as in Theorem \ref{lambda z} that  $U$ is a domain with the property that $\lambda z\in U,$ for each $z\in U$ and each $\lambda$ with $|\lambda|\le 1.$ Then $U$ is a connected Reinhardt domain. Then, $U$ is a connected Reinhardt domain. Now, suppose $f$ is holomorphic in $U.$ We claim that the Laurent series for $f$ is in fact the Taylor series. It is sufficient to show that $c_\nu=0,$ if $\nu_j<0,$ for some $j=1,\cdots,n.$ We may write
$$
	c_\nu(\omega) =  
\frac{1}{(2\pi)^n}\int_{bP^\prime}
	\int_{|\zeta_j|=\rho_j}\frac{f(\omega,\zeta)}{\zeta_j^{\nu_j+1}}d\zeta_j   
\frac{d\zeta'}{\prod_{k\not=j}\zeta_k^{\nu_k+1}},
$$ 
where $P^\prime$ is the polydisc $\prod_{k\not=j}(|z_k|\le \rho_k).$ Since $\nu_j<0,$ 
$$
	\int_{|\zeta_j|=\rho_j}\frac{f(\omega,\zeta)}{\zeta_j^{\nu_j+1}}d\zeta_j  = 0
$$
by Cauchy's Theorem. Thus, $c_\nu=0$ as claimed and the Laurent series for $f$ is indeed the Taylor series. By Theorem \ref{Laurent}, for each $\omega,$ the Taylor series then converges normally to $f(\omega,\cdot)$ on $U.$ It follows that the homogeneous expansion also converges compactly to $f(\omega,\cdot)$ and so the partial sums $p_m(\theta,\cdot)$ converge compactly to $f(\omega,\cdot).$ Thus, we have Theorem \ref{lambda z}.


\section{Uniform algebras}

Let $A$ be a uniform algebra on a compact metric space $K$ (see \cite{G}) and denote by $M$ the maximal ideal space of $A.$  A random element of $A$ is a measurable mapping 
$$
	F:\Omega\rightarrow A, \quad \omega\mapsto F_\omega.
$$

The next property will be of use to us.

\begin{lemma}
	The evaluation homomorphism $\phi_z : A\rightarrow \C$ defined as $\phi_z(f)=f(z)$, for some $z\in K$, is a continuous function.
\end{lemma}

\begin{proposition}\label{random element}
Let $F$ be a random element of $A,$ and define a function
$$
	f:\Omega\times K\rightarrow\C, \quad (\omega,z)\mapsto F_\omega(z).
$$
Then, 
\begin{equation}\label{f to F}
	f(\omega,\cdot)\in A, \quad \forall \omega\in\Omega \quad\quad \mbox{and}\quad\quad
	f(\cdot, z) \quad \mbox{is measurable}, \quad \forall z\in K. 
\end{equation}

Conversely, if a function $f:\Omega\times K\rightarrow\C$  satisfies (\ref{f to F}), then $\omega\mapsto f(\omega,\cdot)$ is measurable. 
\end{proposition}

\begin{proof}
Let $F$ be a random element of $A.$
Since $f(\omega,\cdot)=F_\omega,$  we have that $f(\omega,\cdot)\in A.$ Thus $f$ satisfies the first assertion of  (\ref{f to F}).
For the second assertion of (\ref{f to F}), we note that for $z$ fixed, 
$$f(\omega,z)=F_\omega(z)=\phi_z(F_\omega)=(\phi_z\circ F)(\omega).$$
Thus, $f(\cdot,z)$ is the composition of the measurable function $F$ with the continuous function $\phi_z$. Therefore $f(\cdot,z)$ is indeed measurable. 

Conversely, suppose $f:\Omega\times K\rightarrow\C$  satisfies (\ref{f to F}), with $K$ a compact metric space (such as a compact subset of $\C^n$). We may suppose that $A$ is a uniform algebra over this space. 
We show this in multiple steps.

It is known that $C(K)=C(K,\C)$ is separable. Since $A$ is a subset of $C(K)$, a metric space, it is itself separable, with $Q$ dense in $A$ and countable.

1) Let $f$ be a function respecting $(\ref{f to F})$. We define $h$ : $\Omega \longrightarrow A$ by $h(\omega)=f(\omega,\cdot)$. Denote by $\overline{B_r}(g)=\{j \in A : \sup_{x\in K} |j(x)-g(x)|\leq r\}$ the closed ball of radius $r>0$ around a point of $A$. Finally, we write $f_z^{-1}(E)=\{\omega \in \Omega : f(\omega, z)\in E\}$. We know by hypothesis that if $E$ is measurable, then $f_z^{-1}(E)$ is measurable.

We now have that
\begin{align*}
	h^{-1}(\overline{B_r}(g))&=\{\omega : \sup_{x\in K} |f(\omega,x)-g(x)|\leq r\}\\
	&=\bigcap_{x\in K}\{\omega : |f(\omega,x)-g(x)|\leq r\}.
\end{align*}
We know that $f(\omega,\cdot)$ is continuous, as well as  $g$ and the norm function, and thus by composition $|f(\omega,x)-g(x)|$ is a continuous function of $x$. We know $K$ is separable, with countable dense subset $P=\{x_j\}_{j=0}^\infty$. By continuity and density, $|f(\omega,x)-g(x)|\leq r$ for all $x\in K$ if and only if $|f(w,x_j)-g(x_j)|\leq r$ for all $x_j\in P$. Thus, we can write
\begin{align*}
	h^{-1}(\overline{B_r}(g))&=\bigcap_{x_j\in P}\{\omega : |f(\omega,x_j)-g(x_j)|\leq r\}\\
	&=\bigcap_{x_j\in P}f_{x_j}^{-1}(\{\lambda\in \C : |\lambda-g(x_j)|\leq r\}).
\end{align*}
As the $\{\lambda\in \C : |\lambda-g(x_j)|\leq r\}$ are closed sets, and thus Borel-subsets of $\C$. The pre-images will then be measurable and by countable intersection $h^{-1}(\overline{B_r}(g))$ will be measurable.

2) $h$ is measurable : Let $O$ be an open set of $A$. We recall that $Q$ is a countable dense subset of $A$. As all points of $O$
are interior, we have
\begin{align*}
	O=\bigcup_{g\in Q\cap O} \overline{B_{\frac{d(g,A\setminus O)}{2}}}(g).
\end{align*}
Thus,
\begin{align*}
	h^{-1}(O)&=h^{-1}\left({\bigcup_{g\in Q\cap O} \overline{B_{\frac{d(g,A\setminus O)}{2}}}(g)}\right)\\
	&=\bigcup_{g\in Q\cap O} h^{-1}\left({\overline{B_{\frac{d(g,A\setminus O)}{2}}}(g)}\right)
\end{align*}
and $h^{-1}(O)$ is thus the countable union of measurable sets. As this is true for all open sets  in $A$, and the open sets are a base of the Borel subsets of $A$, the pre-image of $h$ of a Borel subset of $A$ is measurable. As such, $h$ is measurable. By the definition of $h$, we have that $\omega\mapsto f(\omega,\cdot)$ is measurable.
\end{proof}

\begin{corollary}\label{maps F to f}
Let $X$ be a compact metric space. Let
$$
	F:\Omega\rightarrow C(X,\C^n) \quad \omega\mapsto F(\omega)
$$
be a random element of $C(X,\C^n),$ and define a function
$$
	f:\Omega\times X\rightarrow\C^n, \quad (\omega,x)\mapsto F(\omega)(x).
$$
Then, 
$$
	f(\omega,\cdot)\in  C(X,\C^n), \quad \forall	
		\omega\in\Omega \quad\quad \mbox{and}\quad\quad
	f(\cdot, x) \quad \mbox{is measurable}, \quad \forall x\in X. 
$$

Conversely, if a function $f:\Omega\times X\rightarrow\C^n$  satisfies this condition, then $\omega\mapsto f(\omega,\cdot)$ is measurable. 
\end{corollary}

\begin{proof}
For $j=1,\ldots,n,$ 
let $F_j(\omega)$ be the components of $F(\omega)$ and, define functions 
$$
	f_j:\Omega\times X\rightarrow\C, \quad (\omega,x)\mapsto F_j(\omega)(x).
$$
By  Propositions \ref{component} and \ref{cont} of the appendix, each $F_j$ is a random element of $C(X,\C)=C(X).$ By Proposition \ref{random element}, for each $j=1,\ldots n,$ 
$$
	f_j(\omega,\cdot)\in C(X) \quad \forall \omega\in\Omega \quad\quad \mbox{and}\quad\quad
	f_j(\cdot, x) \quad \mbox{is measurable}, \quad \forall x\in X.
$$
Invoking Propositions \ref{component} and \ref{cont} again, 
$$
	f(\omega,\cdot)\in  C(X,\C^n), \quad \forall	
		\omega\in\Omega \quad\quad \mbox{and}\quad\quad
	f(\cdot, x) \quad \mbox{is measurable}, \quad \forall x\in X. 
$$
This proves the first part of the corollary. The proof of the other direction is similar. 
\end{proof}

Suppose  $f:\Omega\times K\rightarrow\C$  satisfies (\ref{f to F}). Then,  on account of Proposition \ref{random element},
we shall by abuse of notation and terminology, say that $f$ is a random element of $A.$ To $f$ we associate the Gelfand transorm 
$$
	\widehat f:\Omega\times M\rightarrow\C, \quad (\omega,\phi) \mapsto \phi(f(\omega,\cdot)).
$$

\begin{proposition}\label{f hat}
Suppose $f$ is a random element of $A.$ Then, 
the Gelfand transform $\widehat f$ is a random (continuous) function on $M.$ 
\end{proposition}

\begin{proof}
We must show that, for each $\phi\in M,$  the function $\widehat f(\cdot,\phi)$ is  measurable. This follows from the fact that  $\widehat f(\cdot,\phi)$ is the composition of the continuous function $\phi$ with the function $\omega\mapsto f(\omega,\cdot),$ which is  measurable, since $f$ is a random element of $A.$

\end{proof}

\section{Random compact sets} 

\color{black}
Because of the Oka-Weil Theorem, the most important notions in complex approximation in several variables are those of polynomial  or rational convexity of compacta.  
\color{black}
First we state a few basic properties if random compact sets and then show that the evaluation of a random function on a compact set is a random compact set. We also show that the polynomially and rationnally convex hulls are transformations that preserve randomness, and that the Siciak extremal function and pluricomplex Green function of a random compact set are random functions.

For a metric space, $(X,d)$ we denote by $(\mathcal K^\prime(X),d_H)$ the space of non-empty compact subsets of $X$ equipped with the Hausdorff distance $d_H.$  We recall the following useful property. 

\begin{lemma}\label{ksep}
    The spaces $\mathcal K^\prime(\R^n)$ and $\mathcal K^\prime(\C^n)$ are separable.
\end{lemma}

Let $g: X \rightarrow Y$ be a continuous function. We may extend it to a function $g^{\mathcal K^\prime}:\mathcal K^\prime(X)\rightarrow \mathcal K^\prime(Y),$  defined by setting $g^{\mathcal K^\prime}(Q)=g(Q),$ for each $Q\in\mathcal K^\prime(X).$ 

\begin{lemma}\label{continuous extension}
Let $X$ and $Y$ be metric spaces. If $g:X\rightarrow Y$ is a continuous function on $X,$ then the extension $g^{\mathcal K^\prime}:\mathcal K^\prime(X)\rightarrow \mathcal K^\prime(Y)$ is a continuous function. 
\end{lemma}

\begin{proof} 
Let $X$ be a metric space (not necessarily compact). 
It is sufficient to show that  $g^{\mathcal K^\prime}$ is continuous at each point of $\mathcal K(X).$ Fix a compact set $Q_0\subset X$ and $\varepsilon>0.$ We claim there is a $\delta>0,$ such that 
\begin{equation}\label{common delta}
	\forall (a,b)\in Q_0\times X, \quad\quad d(a,b)<\delta \Rightarrow d(f(a),f(b))<\varepsilon. 
\end{equation}
For each $a\in Q_0,$ there is a $\delta_a>0,$ such that for all $b\in X,$ if $d(a,b)<\delta_a,$ then $d(f(a),f(b))<\varepsilon/2.$ By compactness of $Q_0,$ there are finitely many points $a_1,\ldots, a_n\in Q_0,$ such that setting $B_j=B_{\delta_j/2}(a_j),$ where $\delta_j=\delta_{a_j},$ we have that $Q_0\subset B_1\cup\cdots\cup B_n.$ Using the triangle inequality, it is easy to see that (\ref{common delta}) holds, with $\delta=(1/2)\min\{\delta_1,\ldots,\delta_n\}.$

From  (\ref{common delta}) it follows that $d_H(Q_0,Q)<\delta \Rightarrow d_H(f(Q_0)),f(Q))<\varepsilon.$ Therefore $f^{\mathcal K^\prime}$ is continuous at $Q_0.$ Since $Q_0\in \mathcal K^\prime(X)$ was arbitrary, it follows that  $f^{\mathcal K^\prime}$ is continuous. 
\end{proof}

\begin{remark}\label{characterizations}
Since $\mathcal K^\prime(X)$ is a metric space, we may also consider it as a measure space, where the measurable sets are the Borel subsets of $\mathcal K^\prime(X).$ 
We may also, as shown in Appendix C of \cite{Mo}, characterize $\mathcal B(\mathcal K^\prime(X))$ as the $\sigma$-algebra generated by the sets $\{K\in \mathcal K^\prime(X) : K\subset G\}$ where  $G$ varies over the open sets of $X$. Alternatively, the Borel sets can be generated by the sets $\{K\in \mathcal K^\prime(X) : K\cap G\neq \emptyset\}$ where again $G$ varies over the open sets of $X$.
\end{remark}

A random compact set in $X$ is a  measurable function $k:\Omega\rightarrow\mathcal K^\prime(X).$ 
If $k_j:\Omega\rightarrow \mathcal K^\prime(X), \, j=1,2,$ are two random compact sets, we denote by  $k_1\cup k_2$ the function $\Omega\rightarrow\mathcal K^\prime(X),$ defined by $(k_1\cup k_2)(\omega)=k_1(\omega)\cup k_2(\omega),$ for $\omega\in\Omega.$ 

By the characterization of Remark \ref{characterizations}, it is easy to prove the following lemmas.

\begin{lemma}\label{cup}
If $k_j:\Omega\rightarrow \mathcal K^\prime(X), \, j=1,2,$ are two random compact sets, then $k_1\cup k_2$ is a random compact set. We also have the countable intersection of random compact sets is a random compact set. 
\end{lemma}

\begin{lemma}\label{countable cup}
Let $\{k_i\}_{i=0}^\infty$ be random compact sets. Suppose we know that for each $\omega$, $k(\omega)=\cup_{i=0}^\infty k_i(\omega)$ is a compact set. Then $k$ is a random compact set.
\end{lemma}

By corollary \ref{maps F to f}, if $X$ is a compact metric space, we can say that a function $f: \Omega \times X \rightarrow \C^n$ is (by abuse of notation and terminology) a random element of $C(X, \C^n)$ if $f(\omega, \cdot) \in C(X, \C^n)$ for all $\omega \in \Omega$ and $f(\cdot, z)$ is measurable for all $z \in X$.

\begin{lemma}\label{singleton}
Let $X$ be a compact metric space and  $f:\Omega\times X\rightarrow \C^n$ be  a random  element of $C(X,\C^n),$ 
\color{black}
Then for each $x\in X,$ the function 
$$
	f(\cdot,x):\Omega\rightarrow\C^n, \quad \omega\mapsto f(\omega,x)
$$ 
is a random complex vector 
and the (singleton-valued) function
$$
	\mathcal K^f(\cdot,\{x\}):\Omega\rightarrow\mathcal K^\prime(\C^n), \quad 
		\omega\mapsto \{f(\omega,x)\}
$$
is a random compact set. 
\end{lemma}

\begin{proof}
The first assertion follows from the definition of a random function. 

For brevity, we write $F_x=\mathcal K^f(\cdot,\{x\}).$ We need to show that $F_x$ is measurable. We use separability arguments. For $W\in \mathcal K^\prime(\C^n)$ and  $r>0,$ consider the closed ball $\overline{B_r}(W)=\{V\in \mathcal K^\prime(\C^n) : d_H(W,V)\leq r\}$  in $\mathcal K^\prime(\C^n)$. 
It is easy to see that every closed ball is closed.

Denote by $\widetilde x:C(X,\C^n)\rightarrow  \mathcal K^\prime(\C^n)$ 
the mapping $g\mapsto \{g(x)\},$ for $g\in C(X,\C^n).$ We then see directly that
\begin{align*}
    \widetilde x^{-1}(\overline{B_r}(W))&=\{g \in C(X,\C^n) : d_H(g(\{x\}),W)\leq r \}\\
    &=\bigcap_{y\in W}\{g \in C(K,\C^n) : d(g(x),y)\leq r \}.
\end{align*}
As $\C^n$ is separable and $W\subset \C^n$, $W$ is also separable. There thus exists a countable dense subset $W^*$ of $W$ and since $d$ is a continuous function, we can restrict the intersection to $W^*$. We also know that the $\{g\in C(X,\C^n) : d(g(x),y)\leq r\}$ are measurable sets as they are closed. As such, $ \widetilde x^{-1}(\overline{B_r}(W))$ is measurable, being the countable intersection of measurable sets. We know from Lemma \ref{ksep} that $\mathcal K(\C^n)$ is separable. Thus, by a similar argument to that in the proof of Proposition \ref{random element}, every open set of $\mathcal K^\prime(\C^n)$ can be expressed as a countable union of closed balls, and we can then generalize to all Borel subsets. Hence,  the function $\widetilde x$ is a measurable function.  
By hypothesis, $f$ is a random element of $C(X,\C^n),$ so the mapping $\omega\mapsto f(\omega,\cdot)$ is measurable. 
It follows that $F_x$ is measurable, since it is  the composition 
$\widetilde x(f(\omega,\cdot))$ of measurable functions: 
$$
	  F_x(\omega) = 
		\mathcal K^f(\omega,\{x\} ) =  \{f(\omega,x)\} =  \widetilde x(f(\omega,\cdot)).
$$

\end{proof}

Let $K$ be a compact subset of $\C^n$ and $A$ a uniform algebra on $K.$ As in \cite{G}, we have that the spectrum of a function $f\in A$ is $\sigma(f)=\{\lambda \in \C : \lambda-f$ is not invertible in $A\}$. We may consider $\sigma=\sigma(f)$ as a mapping
$$	
	\sigma(f):\Omega\rightarrow \mathcal K^\prime(\C), \quad  
		\omega\mapsto \sigma(f(\omega,\cdot)).
$$

\begin{theorem}\label{F sub f}
Suppose $X$ is a compact metric space. If $g:\Omega\times X\rightarrow\C^n$ is a random continuous function on $X$, then $X^g(\omega)$ is a random compact set in $\C^n.$ 
\end{theorem}

\begin{proof}
Let $g:\Omega\times X\rightarrow \C^n$ be a random continuous function. 
We need only show that $X^g$ is  measurable. Let $\{x_1,x_2,\ldots\},$ be a countable (ordered) dense subset of $X;$  let $X_j$ be the set consisting of the first $j$ elements of this set; and let $g_j(\omega,x)$ be the restriction of $g(\omega,x)$ to $\Omega\times X_j.$
Clearly, for each $j,$ the function $g_j$ is a random continuous function on $X_j.$  For each $j=1,2,\cdots,$ we define  the set-function
$$
	F_j:\Omega\rightarrow \mathcal K^\prime(\C^n), \quad  \omega\mapsto g_j(\omega, X_j). 
$$
We shall show that $F_j$ is a random compact set.

By Lemma \ref{singleton}, the function $F_1$ is itself measurable (and analogously for all 1-element sets). By Lemma \ref{continuous extension}, $F_1:\Omega\rightarrow\mathcal K^\prime(\C^n)$ is a continuous random compact set. 

Denote by $C(X,\C^n)$ the set of continuous functions from $X$ to $\C^n$ and,
for $x\in X,$  denote by $\widetilde x:C(X,\C^n)\rightarrow  \mathcal K^\prime(\C^n)$ 
the mapping $h\mapsto \{h(x)\},$ for $h\in C(X,\C^n).$
Again, by Lemma \ref{singleton}, each $\widetilde x_j$ is measurable and then applying Lemma \ref{cup}, $j$ times, we obtain that $F_j$ is a random compact set. 
Since $g(\omega,\cdot)$ is continuous for each $\omega$, its extension $g^{\mathcal K^\prime}(\omega,\cdot)$ to compact sets is continuous by Lemma \ref{continuous extension}. Since $X_j\rightarrow X,$ one can then enter the limit into the function : 
$$
	\lim_{j\rightarrow \infty}F_j(\omega)=
	\lim_{j\rightarrow \infty}g^{\mathcal K^\prime}(\omega, X_j)=
	g^{\mathcal K^\prime}(\omega,\lim_{j\rightarrow \infty}X_j) = 
	g^{\mathcal K^\prime}(\omega,X)=X^g(\omega).
$$ 
We have shown that the sequence  $F_j$ of  measurable functions converges  pointwise to the function $X^g.$  Since these functions take their values in a metric space, it follows from 
Proposition \ref{limsmooth}  
that $X^g$ is  measurable, which concludes the proof. 
\end{proof}

Let $K$ be a compact subset of $\C^n$ and $A$ a uniform algebra on $K.$ If $f:\Omega\times K\rightarrow\C$ is a random element of $A,$ we define the spectrum $\sigma(f)$ to be the set of pairs $\{(\omega, \sigma(f(\omega,\cdot)): \omega\in\Omega\}.$ We may consider $\sigma=\sigma(f)$ as a mapping
$$	
	\sigma(f):\Omega\rightarrow \mathcal K^\prime(\C), \quad  
		\omega\mapsto \sigma(f(\omega,\cdot)).
$$

\begin{proposition}
If $A$ is a uniform algebra for which the maximal ideal space is a separable metric space, then the spectrum $\sigma(f)$ of a random element $f$ of $A$ is a random compact set.
\end{proposition}

\begin{proof}
By hypothesis, we know that $M$ (the maximal ideal space) is a compact metric space.

We wish to show that the mapping $\omega\mapsto \sigma(f(\omega,\cdot))$ is measurable. From Theorem 2.7 of 
\cite{G}, this is equivalent to the mapping $\omega\mapsto \widehat f(\omega,M)$. We know by Proposition \ref{f hat} that $\widehat f$ is a continuous random function over $M$. We remark that the  mapping  $\omega\mapsto \widehat f(\omega,M)$ is in fact the same as the mapping
$\omega\mapsto M^{\widehat f}(\omega).$ By Theorem \ref{F sub f}, $M^{\widehat f}$
is a random compact set, and thus $\sigma(f)$ is a random compact set.
\end{proof}

The preceding proposition applies, for example, if $A=C(K)$ or $A=P(K).$

\begin{proposition}
	For a compact set $K\subset\C^n,$ if $A=C(K)$ or $A=P(K),$  the joint spectrum $\sigma(f_1,\ldots,f_m)$ of finitely many 
random elements
 $f_1,\ldots,f_m$ of $A$  is a random compact set.
\end{proposition}
 
\begin{proof}
	We define the function $\psi : \omega\times M_A\rightarrow \C^m$ by setting $$\psi(\omega,\phi)=(\widehat{f}_1(\omega,\phi),\widehat{f}_2(\omega,\phi),\dots,\widehat{f}_m(\omega,\phi)).$$
	
	For fixed $\phi$, by Proposition \ref{f hat} each $\widehat f_i(\cdot,\phi)$ is a measurable function. From Proposition \ref{component} in the appendix, we have that $\psi(\cdot,\phi)$ is a measurable mapping. For fixed $\omega$, each $\widehat f_i(\omega,\cdot)$ is a continuous function. From Proposition \ref{cont} in the appendix, we have that $\psi(\omega,\cdot)$ is a continuous mapping. Thus, $\psi$ is a random continuous mapping. 
	
	From the definition of the joint spectrum \cite{G} we have 
that $\sigma(f_1,\dots,f_m)(\omega)=\psi(\omega,M_A)$. We see that the  mapping $\omega \mapsto \psi(\omega, M_A)$ is in fact the mapping $\omega \mapsto M_A^{\psi}(\omega)$. We now apply Theorem \ref{F sub f}, and we have that the joint spectrum is a random compact set.
\end{proof}

We define a compact transformation as a function $T : \mathcal K'(X)\rightarrow \mathcal K'(X)$. We say that a compact transformation is randomness-preserving if 
$K$ being a random compact set implies that $T(K)$ is a random compact set. It is important that this property not depend of a specific structure of $\Omega$, it must work for any measurable space.

\begin{lemma}\label{randomness-preserving}
    A compact transformation is randomness-preserving if and only if it is a measurable function.
\end{lemma}
\begin{proof}
    Suppose a compact transformation $T$ is randomness-preserving. This means that for any choice of random events set $(\Omega,\mathcal A)$, the measurability of the mapping $\omega \mapsto K(\omega)$ implies the measurability of the mapping $\omega\mapsto T(K(\omega))$.
    
We may thus choose $\Omega=\mathcal K'(X)$ and $\mathcal A=\mathcal B(\mathcal K'(X))$. We now look at the 
identity mapping $I:\Omega\rightarrow \mathcal K'(X)$  mapping $K$ to $K.$ This mapping defines a random compact set, since the identity function is certainly measurable. 
Thus, as $T$ is randomness-preserving, the mapping $K\mapsto T(I(K))=T(K)$ is measurable.

    The converse follows directly by composition of measurable functions.
\end{proof}

For $K\in \mathcal K^\prime(\C^n),$ we denote by $\widehat K$ the polynomially convex hull of $K.$ It is defined as $$\widehat K = \{z \in C^n : \abs{p(z)} \leq \max_{x\in K} \abs{p(x)} \text{ } \forall p \in P(\C^n)\}$$where $P(\C^n)$ is the set of complex polynomials from $\C^n$ to $\C$. We would like to show that the polynomially convex hull of a random compact set is also a random compact set.

\begin{lemma} \label{lemmaPolyHull}
Let $K$ be a compact set of $\C^n$ and let $\mathcal P^\Q(\C^n)$ denote the polynomials in $\C^n$ whose coefficients have rational real and imaginary part. Then,
\begin{equation} \label{hullQ}
    \widehat K = \{z \in C^n : \abs{p(z)} \leq \max_{x\in K} \abs{p(x)} \text{ } \forall p \in P^\Q(\C^n)\}.
\end{equation}
\end{lemma}

\begin{proof}
    Let $\varepsilon > 0$ and let $\widehat K_\Q$ denote the right side of equation \ref{hullQ}. From the definition of $\widehat K$, it is clear that $\widehat K \subset \widehat K_\Q$. We wish to prove that $\widehat K_\Q \subset \widehat K$. Let $P$ be a polynomial of degree $N$ which we write as
    \begin{equation*}
        P(z) = \sum_{k=1}^N \sum_{\abs{\alpha} = k} a_\alpha z^\alpha.
    \end{equation*}
    Let $M_\alpha$ denote the maximum of $\abs{z^\alpha}$ over all $z$ in $K$. By the density of the rationals, there exists $b_\alpha$ with rational real and imaginary part such that, 
if $|\alpha|=k$ and $\mu(k)$ is the number of indices $\alpha$ with $|\alpha|=k,$ then
    \begin{equation*}
        \abs{a_\alpha - b_\alpha} < \frac{\varepsilon}{N\mu(k)(M_\alpha + 1)}.
    \end{equation*}
    Let $Q$ denote the polynomial with coefficients $b_\alpha$. Thus,
    \begin{align*}
        \abs{P(z) - Q(z)} &= \abs{\sum_{k=1}^N \sum_{\abs{\alpha} = k} (a_\alpha - b_\alpha) z^\alpha}\\
        &\leq \sum_{k=1}^N \sum_{\abs{\alpha} = k} \abs{a_\alpha - b_\alpha} \abs{z^\alpha}\\
        &<\sum_{k=1}^N \sum_{\abs{\alpha} = k} \frac{\varepsilon}{N\mu(k)(M_\alpha + 1)} \abs{z^\alpha}\\
       &< \varepsilon
    \end{align*}
    and so we conclude that for every polynomial $P$ and $\varepsilon > 0$, there exists a polynomial $Q$ in $\mathcal P^\Q(\C^n)$ such that $\abs{P(z) - Q(z)} < \varepsilon$ for all $z \in K$.
    
    Let $z_0$ be a point in $\widehat K^c$. Therefore there exists a polynomial $P_0$ with maximum in $K$ denoted by $M$ such that $\abs{P_0(z_0)} > M$. This means there exists a $\eta > 0$ such that $\abs{P_0(z_0)} = M + \eta$. By the previous explanation but applied to the compact set $K \cup \{z_0\}$, there exists a $Q_0$ in $\mathcal P^\Q(\C^n)$ such that $\abs{P_0(z) - Q_0(z)} < \eta/4$. Thus, using the reverse triangle inequality, we have that 
    \begin{equation*}
        \abs{Q_0(z_0)} > M + \frac{3\eta}{4} > \max_{x \in K} \abs{Q_0(x)} + \frac{\eta}{2} > \max_{x \in K} \abs{Q_0(x)}
    \end{equation*}
    and so $z_0$ is in $\widehat K_\Q^c$. 
\end{proof}

We denote by $\mathcal K(X)$ the set $\mathcal K^\prime(X)\cup\{\emptyset\}$ of {\em all} (including the empty set) compact subsets of $X.$ 
Let $W$ be a measurable space.
We define as a pseudo-random compact set in $X$ a mapping 
$$f : W \rightarrow\mathcal K(X),$$ 
if pre-images of Borel subsets of $\mathcal K^\prime(X)$ are measurable. 
Here $W$ can be a metric space with the Borel sets, or $W=\Omega$ . 
We see that every measuralble compact set in $\C^n$ is a pseudo-random compact set. If $f$ is a pseudo-random compact set and $f^{-1}(\{\emptyset\})$ is empty, then $f$ is a measurable compact set (here we see $f^{-1}(\{\emptyset\})$ as the pre-image of the empty compact set, and not as the the pre-image of an empty set of compact sets). 
We have directly that the composition of a measurable function followed by
a pseudo-random compact set is a pseudo-random compact set.

We define as a pseudo-continuous compact-valued function a function $f: W\rightarrow\mathcal K(\C^n)$, where $W$ is a metric space, if for all $w\in W$ such that $f(w)\neq \emptyset$, and $\varepsilon>0$, there exists $\delta>0$ such that if $d(w,y)<\delta$ and $f(y)\neq \emptyset$, then $d_H(f(w),f(y))<\varepsilon$, and if moreover the set $f^{-1}(\emptyset)$ is a measurable set.

\begin{lemma}\label{pseudo union}
	If $\{K_i\}_{i=0}^\infty$ be a sequence of pseudo-random compact sets, then $k(\omega)=\bigcap_{i=0}^\infty K_i(\omega)$ is a pseudo-random compact set. If we know that $K(\omega)=\bigcup_{i=0}^\infty K_i(\omega)$ is a compact set for each $\omega$, then $K$ is a pseudo-random compact set.
\end{lemma}

\begin{proof}
	This proof follows from the fact that we supposed that the pre-images of Borel subsets are measurable. We may use a similar proof method to Lemmas \ref{cup} and \ref{countable cup}.

\end{proof}

\begin{lemma}\label{pseudo cont}
Let $f : X\rightarrow\mathcal K(\C^n)$ be a pseudo-continuous compact-valued function. Then it is also a pseudo-random compact set.
\end{lemma}

\begin{proof}
    By our definition of pseudo-continuous function, we have that the restriction $f|_{X\setminus f^{-1}(\emptyset)}$ is a continuous function over $X\setminus f^{-1}(\emptyset).$
Thus, if $O$ an open set of $\mathcal B(\C^n),$ then $f^{-1}(O)$ is an open set relative to $X\setminus f^{-1}(\emptyset)$. 
From the definition of pseudo-continuity, 
$f^{-1}(\emptyset)$ is measurable in $X,$ so $X\setminus f^{-1}(\emptyset)$ is also measurable. We have that  $f^{-1}(O)$ is a measurable subset of the measurable set  $X\setminus f^{-1}(\emptyset).$ 
Thus, $f^{-1}(O)$ is measurable relative to $X.$
As the open sets generate the Borel subsets, all pre-images of measurable subsets of $\mathcal K(\C^n)$ are measurable.
\end{proof}

\begin{theorem}\label{P-hull random}
Let $K$ be a random compact set of $\C^n$. Then, its polynomially convex hull $\widehat K$, defined pointwise as $\widehat K(\omega)=\widehat{K(\omega)}$, is a random compact set.
\end{theorem}

\begin{proof} 
   
    Let $N_p^K = \{z : \abs{p(z)} \leq \max_{x \in K} \abs{p(x)}\}$. Then, by  Lemma \ref{lemmaPolyHull},
    \begin{equation*}
        \widehat K(\omega) = \bigcap_{p \in P^\Q(\C^n)} N_p^{K(\omega)}.
    \end{equation*}
    The function $|p(\cdot)|$ is  continuous and thus
$$N_p^{K(\omega)}=|p(\cdot)|^{-1}(\overline B_{\max_{x \in K(\omega)} \abs{p(x)}}(0))$$
is a closed set as the pre-image of a closed set. But $N_p^K$ is not necessarily bounded. To solve this, we may write :
      \begin{equation}\label{khat}
        \widehat K(\omega) = \bigcup_{i=1}^\infty \bigcap_{p \in P^\Q(\C^n)} 						\left(N_p^{K(\omega)}\cap \overline{B}_i(0)\right).
    \end{equation}
This is true since for every $\omega$, if $i$ is  great enough, the compact set $\widehat K(\omega)$ is contained in the ball of radius $i$. Each $N_p^{K(\omega)}\cap \overline{B}_i(0)$ is compact, since it is  the intersection of a closed set with a compact set.
    
    Choose a rational non-constant polynomial $p$ and a natural $i$. We shall show that the mapping $\omega\mapsto N_p^{K(\omega)}\cap \overline{B}_i(0)$ is a pseudo-random compact set. It will follow from Lemma \ref{pseudo union} that $\bigcap_{p \in P^\Q(\C^n)} (N_p^{K(\omega)}\cap \overline{B}_i(0))$ is a pseudo-random compact set, since the intersection is countable ($P^\Q(\C^n)$ is in bijection with the set $\mathcal V$ where $n = 2$ in the proof of Lemma \ref{ksep}). Then, as we already know that $\widehat K(\omega)$ is compact and non-empty, we can apply Lemma \ref{pseudo union} to prove the theorem.
    
   We consider several auxiliary  mappings. We already know the mapping $\omega\mapsto K(\omega)$ is measurable, by hypothesis.
As $|p(\cdot)|$ is continuous, we know from Lemma \ref{continuous extension} that $K\mapsto |p|(K)$ is a continuous mapping from $\mathcal K^\prime (\C^n)$ to $\mathcal K^\prime(\R)$. Also, the mapping $K\mapsto \max\{x:x\in K\}$ between $\mathcal K^\prime(\R)$ and $\R$ is also continuous. Let $\varepsilon>0$ and fix $K_0\in \mathcal K^\prime(C^n)$. Suppose $d_H(K_0,K) < \varepsilon$. Then, for all $x \in K_0$ there is a $y \in K$, and for all $y \in K$ there is a $x \in K_0$, such that in both cases $\abs{x-y} < \varepsilon$. It follows that
    \begin{equation*}
        |\max\{x:x\in K_0\}-\max\{x:x\in K\}|<\varepsilon.
    \end{equation*}

Denote $W_p^M=\{z : |p(z)|\leq M\}$. We shall now show that the mapping 
$W: \R \rightarrow \mathcal{K}(\C^n)$ which maps $M$ to $W_p^M \cap \overline B_i(0)$ is pseudo-random.
Let $O$ be an open set of $\C^n$. Denote as before, for $O$ open in $\C^n$, the sets 
\begin{equation*}
H_{O} = \{K \in \mathcal{K}^\prime(\C^n) : K \cap O\neq \emptyset \},
\end{equation*}
which generate the Borel subsets of $\mathcal K^\prime(\C^n)$. 
Since $W(M) \subset W(N)$ whenever $M \leq N$ (by continuity of $\abs{p}$), we have that
\begin{equation*}
\{M : W(M) \in H_{O}\} = \{M : W(M) \cap O\neq \emptyset\} \in \{[M_0, \infty), (M_0, \infty),\emptyset\},
\end{equation*}
for some $M_0 \in \R$. The mapping $W$ is pseudo-random as a mapping from $\R$ to $\mathcal K(\C^n).$

We can now see that the mapping $\omega \mapsto N_p^{K(\omega)}\cap \overline B_i(0)$ is in fact the sequence of mappings
    \begin{equation*}
        \omega\mapsto K(\omega)\mapsto |p(K(\omega))|\mapsto \max\{x \in |p(K(\omega))|\}\mapsto W_p^{\max\{x \in |p(K(\omega))|\}}\cap \overline B_i(0)
    \end{equation*}
    and is therefore pseudo-random. Being always non-empty, it is random.
    
As noted earlier, this proves the theorem.
\end{proof}

\begin{corollary}\label{P-hull measurable} 
The mapping 
$$
	\mathcal K^\prime(\C^n)\rightarrow\mathcal K^\prime(\C^n), \quad K\mapsto \widehat K
$$
is measurable.
\end{corollary} 

\begin{proof} This follows from Lemma \ref{randomness-preserving} and Theorem \ref{P-hull random}.
\end{proof}

For a random compact set $K:\mathcal A\rightarrow\mathcal K^\prime(\C^n),$ we define the random polynomially convex hull as the function $R\widehat K:\mathcal A\rightarrow \mathcal K^\prime(\C^n),$ given by setting 
$$
	R\widehat K(\omega)=\{z: |p(\omega,x)|\le 
		\max_{x\in K(\omega)}|p(\omega,x)|, \forall \, \mbox{random polynomials}\, p\}.		
$$
\color{black}
Because of the  central role of the polynomially convex hull in the Oka-Weil Theorem, the random polynomially convex hull should be investigated in relation to approximation by random polynomials. 
\color{black}

\begin{proposition}
    Let $K$ be a random compact set. Then, $\widehat K=R\widehat K$.
\end{proposition}
\begin{proof}
    Let $p$ be a polynomial. It is also a random polynomial, by defining $p(\omega,z)=p(z)$ for all $\omega$. For fixed $z$, $p(\cdot,z)$ is a constant function. Thus, if $E$ is a measurable set, we have that $$p(\cdot,z)^{-1}(E)=\begin{cases}
		\Omega &\quad\text{if}\quad p(z)\in E\\
		\emptyset &\quad\text{if}\quad p(z)\notin E\\
		\end{cases}$$
and as both cases are measurable sets, this function is measurable.

    Thus, when filtering out points to form the random polynomially convex hull, all possible polynomials will be applied. But since all random polynomials take as value, for fixed $\omega$, a polynomial, filtering will be done by all polynomials but only by polynomials. Thus, we have directly that $\widehat K=R\widehat K$.
\end{proof}

Let us denote by $ \mathcal R\mbox{-hull}\,K$ the rationally convex hull of a compact set $K.$ It is defined as$$
\mathcal R\mbox{-hull}\,K = \{z \in \C^n : \abs{r(z)} \leq \max_{x \in K}  \abs{r(x)} \text{ } \forall r \in \mathcal R_K(\C^n)\}$$
where $R_K(\C^n)$ is the set of rational functions from $\C^n$ to $\C$ which are analytic over $K$.

\begin{lemma}
Let $K$ be a compact set of $\C^n$ and let $\mathcal R_K^\Q(\C^n)$ denote the rational functions in $\C^n$ without singularities on $K$ and whose coefficients have rational real and imaginary part. Then,
\begin{equation} \label{rationalHullQ}
    \mathcal R\mbox{-hull}\,K = \{z \in \C^n : \abs{r(z)} \leq \max_{x \in K}  \abs{r(x)} \text{ } \forall r \in \mathcal R_K^\Q(\C^n)\}.
\end{equation}
\end{lemma}

\begin{proof}
Let $\mathcal R_\Q\mbox{-hull}\,K$ denote the right side of equation \ref{rationalHullQ}. It is clear that $\mathcal R\mbox{-hull}\,K \subset\mathcal R_\Q\mbox{-hull}\,K$ by the definition. Since every rational function with no singularities on $K$ can be approximated on $K$ by functions in $\mathcal R_K^\Q(\C^n),$ the inclusion  $\mathcal R_\Q\mbox{-hull}\,K \subset \mathcal R$-hull$K$ then follows by the same arguments used at the end of the proof of Lemma \ref{lemmaPolyHull}.

\end{proof}

Note : The result indeed follows from the proof of lemma \ref{lemmaPolyHull} since if $z_0 \in  (\mathcal R\mbox{-hull}\,K)^c$, we can find a rational function $r$ without pole in $K \cap  \{z_0\}$ such that $\abs{r(z_0)} > \max_{x \in K} \abs{r(x)}$. That is because if $r = 1/q$ had a pole at $z_0$, we can nudge it a bit to get the desired rational function by taking $r(z) =1/q(z+v)$ for $v$ sufficiently small.

The following lemma is a known fact that we simply recall.

\begin{lemma} \label{pole set}
	Let $r : \C^n\rightarrow \C$ be a  rational function. Then, the set $S(r)$ of singularities of $r$ is a closed set.
\end{lemma}

The proof of the next lemma is very simple to do, and is thus left to the reader to verify.

\begin{lemma} \label{restriction measurable}
	Let $X$ be a metric space. Suppose  $X=A\cup B,$ where $A$ and $B$ are disjoint measurable subsets. Let $Y$ be a metric space and $f : X\rightarrow Y$ a function such that its restriction $f_A$ to the metric space $A$ is measurable, and its restriction $f_B$ to the metric space $B$ is measurable. Then $f$ is a measurable function.
\end{lemma}

\begin{theorem}\label{R-hull random}
Let $K$ be a random compact set in $\C^n$. Then, its rationally convex hull $\mathcal R\mbox{-hull}\,K$, defined pointwise as $[\mathcal R\mbox{-hull}\,K](\omega) =  \mathcal R\mbox{-hull}\,(K(\omega)),$ is a random compact set.
\end{theorem}

\begin{proof}
Let $r$ be a rational function and $K$ be a non-empty compact subset of $\C^n$. 
Set $A_r=\{K\in \mathcal K^\prime(\C^n) : K\cap S(r)\neq \emptyset\}$ and $B_r=\mathcal K^\prime(\C^n)\setminus A_r.$  Then, $\mathcal K^\prime(\C^n)$ is the disjoint union of $A_r$ and $B_r.$ 
We define a function $\mu_r : \mathcal K'(\C^n) \rightarrow \mathcal K'(\C^n)$ by
\begin{equation*}
    \mu_r(K) =
    \begin{cases}
    \widehat K & \text{if } K\in A_r \\
    \widehat K \cap N_r^K & \text{if } K\in B_r,
    \end{cases}
\end{equation*}
where $N_r^K = \{z \in \C^n : \abs{r(z)} \leq \max_{x \in K} \abs{r(x)}\}$. 
We note that, for $K\in B_r,$ the set $N_r^K$ is never empty, since it contains $K.$ 
Since all polynomials are also rational functions, $\mathcal R\mbox{-hull}\,K\subseteq \widehat K.$
Therefore, we can write
\begin{equation*}
     \mathcal R\mbox{-hull}\,K =  (\mathcal R\mbox{-hull}\,K)\cap\widehat K = 
	\bigcap_{r \in \mathcal R^\Q(\C^n)} \mu_r(K),
\end{equation*}
where $\mathcal R^\Q(\C^n)$ is the set of all rational functions in $\C^n$ whose coefficients have rational real and imaginary parts. The intersection is countable since $\mathcal R^\Q(\C^n)$ can be viewed as a subset of $\mathcal P^\Q(\C^n) \times \mathcal P^\Q(\C^n)$.  

It is thus sufficient to show that $K\mapsto \mu_r(K)$ is measurable, for then by Lemma \ref{cup}, 
the mapping $\omega \mapsto  \mathcal R\mbox{-hull}\,K(\omega)$ will be measurable.

First of all, let us show that the subset $A_r\subset \mathcal K^\prime (\C^n)$  is measurable. 
We may assume that $A_r\not=\emptyset$ and hence $S(r)\not=\emptyset,$  
We construct a sequence of open subsets of $\C^n$ defined as 
$$O_n=\left\{z : d(z,S(r))<\frac{1}{n}\right\}.$$
Then, $O_{n+1}\subseteq O_n$ and  $\bigcap_{n=1}^\infty O_n = S(r)$. We have that$$A_r=\bigcap_{n=1}^\infty \{K : K\cap O_n\neq \emptyset\}.$$ By our 2nd characterisation of measurable sets (see Remark \ref{characterizations}), this is a countable intersection of measurable sets, and thus measurable. We have shown that $A_r$ and hence also $B_r=\mathcal K^\prime(\C_n)\setminus A_r$ are measurable.

The definition of the  function $\mu_r(K)$ depends on whether $K$ is in $A_r$ or $B_r.$ 
Since $\mathcal K^\prime(\C^n) = A_r\cup B_r$ is the union of disjoint measurable sets, in order to assure that $\mu_r$ is measurable, it suffices, by Lemma \ref{restriction measurable}, to show that the restrictions of $\mu_r$ to  $A_r$ and  to $B_r$ are measurable . 

By Corollary \ref{P-hull measurable},
 we have that the mapping $K\mapsto \widehat K$ is measurable, and thus its restriction to $A_r,$ which is the same as the restriction of $\mu_r$ to $A_r,$ is also measurable. 

We claim that the mapping $K\mapsto \widehat K\cap N_r^K$, is measurable over $B_r$. We first notice that over the elements of $B_r$, our rational function $r$ is in fact continuous.
One might by essentially the same proof as for the polynomially convex hull , as it did not use any special properties other than that polynomials are continuous functions, show that (over $B_r$), the mappings $K\mapsto N_r^K \cap \overline B_i(0)$ are pseudo-random. It then follows from Lemma \ref{pseudo union}, as the mapping $K\mapsto \widehat K$ is measurable,  that each mapping $K\mapsto N_r^K \cap \overline B_i(0)\cap \widehat K$ is pseudo-random. We then have that the mapping $$
	K\mapsto N_r^K \cap  \widehat K=\bigcup_{i=1}^\infty N_r^K \cap \overline B_i(0)\cap \widehat K
 $$
is pseudo-random by Lemma \ref{pseudo union}
and in fact measurable, as $N_r^K \cap  \widehat K$ 
is always non-empty.

By Lemma \ref{restriction measurable}, the mapping $\mu_r$ is measurable. 
\end{proof}

\color{black}
	One might also be interested in using random compact sets to construct random functions.
\color{black}	
	An important function for complex approximation (see \cite{Ni}) is the pluricomplex Green function $V_K$ for a non-empty compact set $K\subset\C^n$ is defined as $V_K(z)=\log \Phi_K(z),$ for $z\in\C^n,$ where
\color{black}
$\Phi$ is the  Siciak extremal function, defined as
	\begin{equation*}
		\Phi_K(z)=\sup_{p\in P(\C^n) : \|p\|_K\leq 1, \deg(p)\geq 1} |p(z)|^{1/\deg p}.
	\end{equation*}

We have the following result. 
\color{black}	
	\begin{theorem}\label{siciak}
		The Siciak extremal function and the pluricomplex Green function of a random compact set $K$ are random functions. That is, the functions $(\omega,z)\mapsto\Phi_{K(\omega)}(z)$ and  $V_{K(\omega)}(z)=\log \Phi_{K(\omega)}(z),$ are random function into $\R\cup\{\infty\}$.
	\end{theorem}

	\begin{proof}
It is sufficient to show that the Siciak Green function is a random function. 
\color{black}
		Let $p$ be a polynomial. We construct the function
		\begin{align*}
			g_p(\omega,z)
			\begin{cases}
       			|p(z)|^{1/\deg p} &\quad \|p\|_{K(\omega)}\leq 1\\
       			0 &\quad \|p\|_{K(\omega)}>1,\\
     		\end{cases}
		\end{align*}
		which evaluates our polynomial if it respects the condition $\|p\|_{K(\omega)}\leq 1$, and returns $0$ if not.
		
		We first prove this function is a random function. We fix $z\in \C^n$. Let $M$ be a measurable set of $\R$. We now have
		\begin{align*}
			\{\omega : g_p(\omega,z)\in M\} & =(\{\omega : \|p\|_{K(\omega)}>1\}\cap\color{black}\{\omega : 0\in M\})\\
&\quad\cup (\{\omega : \|p\|_{K(\omega)}\leq1\|\cap\color{black}\{\omega : |p(z)|^{1/\deg p}\in M\}).
		\end{align*}
		We remark that the sets $\{\omega : 0\in M\}$ and $\{\omega : |p(z)|^{1/\deg p}\in M\}$ are always $\emptyset$ or $\Omega$, as the conditions do not depend on $\omega$. The set $\{\omega:\|p\|_{K(\omega)}\le 1\}$ is measurable, because the function $|p(K(\omega))|$ is measurable, since it is the composition of the measurable function $K:\Omega\rightarrow\mathcal K^\prime(\C^n)$ with the continuous function $\mathcal K^{|p|}:\mathcal K^\prime(\C^n)\rightarrow \mathcal K^\prime(\R)$. Thus, the set $\{\omega : g_p(\omega,z)\in M\}$ is measurable. This proves $g_p(\cdot,z)$ is measurable, which means that $g_p$ is a random function.
		
		Let us now return to the Siciak extremal function. We shall prove that it may be written as
		\begin{equation*}
			\Phi_E(z)=\sup_{p\in P^\Q(\C^n) : \|p\|_E\leq 1, \deg(p)\geq 1} |p(z)|^{1/\deg p},
		\end{equation*}
		where $P^\Q(\C^n)$ is the set of all polynomials with rational coefficients.
		
		In order to do this, it is sufficient to show that for any $\varepsilon>0$, $p\in P(\C^n)$, $z\in \C^n$, there is a polynomial $q\in P^\Q(\C^n)$ such that $||p(z)|^{1/\deg p}-|q(z)|^{1/\deg q}|<\varepsilon$. Suppose $p$ has degree $k$. We will choose $q$ of degree $k$ also. Using the fact that if $x>y\geq 0$, then $x^{1/k}-y^{1/k}\color{black}\le\color{black}(x-y)^{1/k}$, and the reverse triangle identity $||x|-|y||\leq|x-y|$ (this one also applies to the complex norm), we have that
		\begin{align*}
			||p(z)|^{1/\deg p}-|q(z)|^{1/\deg q}| &= ||p(z)|^{1/k}-|q(z)|^{1/k}|\\
			&\leq ||p(z)|-|q(z)||^{1/k}\\
			&\leq |p(z)-q(z)|^{1/k}.
		\end{align*}
		This reduces the problem to finding $q$ such that $|p(z)-q(z)|<\color{black}\varepsilon^k$. But we know this is possible, as we can approximate $p$ by choosing $q$ with very close rationnal coefficients.
		
		But now, by the definition of $g_p$, we have that
		\begin{align*}
			\Phi_{K(\omega)}(z)=\sup_{p\in P^\Q(\C^n) : \deg(p)\geq 1} g_p(\omega,z).
		\end{align*}
		We were able to remove the condition over our polynomials that $\|p\|_{K(\omega)}\leq 1$ since $g_p(\omega,z)$ always returns $0$ if the condition is not fulfilled.
		
		We now prove that this function is random; that is,  that $\Phi_{K(\cdot)}(z)$ is measurable. Consider the subset $(-\infty,a]$  of $\R$. We have that
		\begin{align*}
			\{\omega : \Phi_{K(\cdot)}(z)\in (-\infty,a]\}
			&=\{\omega : \Phi_{K(\cdot)}(z)\leq a\}\\
			&=\{\omega : \sup_{p\in P^\Q(\C^n) : \deg(p)\geq 1} g_p(\omega,z)\leq a\}\\
			&\color{black}=\{\omega :g_p(\omega,z)\leq a, \forall p\in P^\Q(\C^n), \deg(p)\geq 1\} \\
			&=\bigcap_{p\in P^\Q(\C^n) : \deg(p)\geq 1}\{\omega : g_p(\omega,z)\leq a\}.
		\end{align*}
		As $g_p$ is a random function, we know that the sets $\{\omega : g_p(\omega,z)\leq a\}$ are measurable. As $\{p\in P^\Q(\C^n) : \deg(p)\geq 1\}$ is a countable set, this is a countable intersection of measurable sets, and is this a measurable set. Thus, $\Phi_{K(\cdot)}(z)$ is a random function.
		
	\end{proof}


\section{Generalized random functions} \label{GeneralizedRandomFunctions}

Let $K$ be a random compact set in $\C^n$. Let the graph of $K$ be defined as 
\color{black}
 in \cite{AB1984}:
\color{black}
\begin{equation*}
\Gr K = \{(\omega, z) \in \Omega\times \C^n : z \in K(\omega)\}.
\end{equation*}
If we consider a compact set $K\subset\C^n$ as a random compact set $K(\omega)=K,$ for all $\omega\in K,$ then $\Gr K=\Omega\times K.$

By abuse of notation, we set 
\begin{equation*}
K^{-1}(z) = \{\omega : z \in K(\omega)\} = \bigcap_{n=1}^\infty \{\omega : K(\omega) \cap B_{\frac{1}{n}}(z) \neq \emptyset\}
\end{equation*}
and note that this does not correspond to the pre-image of the compact set $\{z\}$. 
By the second equality, we see that this set is measurable.
We define a generalized random function as a function $f : \Gr K \rightarrow \C$ 
(hence $f(\omega,\cdot)$ is defined over $K(\omega)$), 
such that, for fixed $z\in \C^n$ with $K^{-1}(z) \neq \emptyset$, 
the function $f(\cdot,z)$ from $K^{-1}(z)$ to $\C$ is measurable. 
\color{black}
In case a compact set $K$ is considered as a random compact set, then a generalized random function on $\Gr K=\Omega\times K$ turns out to be just a random function on $K.$ 
\color{black}

We can now formulate a weaker but random version of Roth's fusion lemma \cite{Gai}. We were not able to prove it without our conjectured random version of Runge's theorem. Therefore, we formulate it as a conjecture,  
\color{black}
but it can be shown that it would follow directly from Conjecture \ref{conjRunge}, if the latter were true. 
\color{black}

\begin{conjecture}\label{fusion}
Let $K_1$, $K_2$ and $k$ be random compact sets of $\C$ such that $K_1$ and $K_2$ are uniformly disjoint 
($K_1(\omega)\cap K_2(\omega)=\emptyset$ for all $\omega$). Furthermore, suppose 
\begin{equation*}
\abs{K_1(\Omega)} + \abs{K_2(\Omega)} + \abs{k(\Omega)} \leq \aleph_0.
\end{equation*}
If there exists generalized random rational functions $r_1,r_2 : \Omega \times \C \rightarrow \overline{\C}$ such that $r_i(\omega, \cdot)$ has no pole in $K_i(\omega)$ for $i = 1, 2$ and if there exists a positive error function $\varepsilon: \Omega \rightarrow \R^+$ such that
\begin{equation*}
\norm{r_1(\omega, \cdot) - r_2(\omega, \cdot)}_{k(\omega)} < \varepsilon(\omega),
\end{equation*}
then there exists a generalized random rational function $r : \Omega\times\C \rightarrow \overline{\C}$ and a measurable positive function $A: \Omega\rightarrow\R^+$ depending solely on $K_1$ and $K_2$ such that
\begin{equation*}
\norm{r(\omega, \cdot) - r_1(\omega, \cdot)}_{(K_1\cup k)(\omega)} < A(\omega)\varepsilon(\omega)
\end{equation*}
and
\begin{equation*}
\norm{r(\omega, \cdot) - r_2(\omega, \cdot)}_{(K_2\cup k)(\omega)} < A(\omega)\varepsilon(\omega)
\end{equation*}
for all $\omega \in \Omega$.
\end{conjecture}

We say a random compact set $K$ is uniformly separable if there exists a countable subset $E$ of $\C^n$ whose intersection with 
$K(\omega)$ is dense in $K(\omega)$ for every $\omega \in \Omega$. In a way it is a generalization of random compact sets taking countably many values. In fact, we have the following proposition.

\begin{proposition}
Let $K$ be a random compact set. If $K$ takes countably many values or if
\begin{equation}\label{unifsep}
\abs{\bigcup_{\omega \in \Omega} K(\omega)\setminus \overline{\intr K(\omega)}} \leq \aleph_0.
\end{equation}
then $K$ is uniformly separable. Moreover, 
if $K(\omega)$ is the closure of a bounded open subset of $\C^n$ for all $\omega \in \Omega$, then $K$ is uniformly separable.
\end{proposition}

\begin{proof}
Let $K$ be a random compact set. If $K$ takes countably many values, for each $\omega,$ choose a countable dense subset $E(\omega)$ of $K(\omega),$ and set $E=\cup_\omega E(\omega).$ This set is countable since it can be viewed as a countable union of countable sets and for every $\omega$, its restriction to $K(\omega)$ is dense in $K(\omega)$. $K$ is therefore uniformly separable.

Let now $K$ be a random compact set satisfying \eqref{unifsep}. Let $T$ be a countable dense subset of $\C^n$. Take $E$ as
\begin{equation*}
E = \left(\bigcup_{\omega \in \Omega} U(\omega)\right) \cup T.
\end{equation*}
where $U(\omega) = K(\omega)\setminus \overline{\intr K(\omega)}$.
By hypothesis, $E$ is a countable subset of $\C^n$. Fix $\omega \in \Omega$ and let $z$ be a point in $K(\omega)$. We wish to prove that $z$ is in the closure of $K(\omega) \cap E$. By using that $U(\omega) \subset K(\omega)$, we have that
\begin{equation*}
K(\omega) \cap E = K(\omega) \cap \left[\left(\bigcup_{\omega \in \Omega} U(\omega)\right) \cup T\right] \supset U(\omega) \cup (K(\omega) \cap T)
\end{equation*}
and therefore it is sufficient to prove that $z$ is in the closure of either $U(\omega)$ or $K(\omega) \cap T$.

If $z$ is an isolated point, then $z \not\in \overline{\intr K(\omega)}$, which means that $z \in U(\omega)$ and is therefore in $\overline{U(\omega)}$.

If $z$ is a limit point, it is either the limit of interior points or not. If it is the limit of interior points, we can find a sequence $\{z_j\}$ in $T \cap K(\omega)$ such that $z_j \rightarrow z$. Thus $z$ is in the closure of $T \cap K(\omega)$. If it is not the limit of interior points, then $z \not\in \overline{\intr K(\omega)}$ and therefore $z \in U(\omega) \subset \overline{U(\omega)}$.

In every case, we have that $z$ is in the closure of $U(\omega)$ or $K(\omega) \cap T$. It follows that $K$ is uniformly separable.

The last statement of the proposition is obvious.
\end{proof}

\begin{theorem}
Let $K$ be a uniformly separable random compact set. Let ${f_j}$ be a sequence of generalized random functions such that
\begin{equation*}
    \forall \omega \in \Omega, f_j(\omega, \cdot) \xrightarrow{unif} f(\omega, \cdot) \text{ on } K(\omega).
\end{equation*}
Then, there exists a sequence ${F_j}$ of generalized continuous functions such that
\begin{equation*}
    F_j \xrightarrow{unif} f \text{ on } \Gr K
\end{equation*}
and such that for all $\omega \in \Omega$ and $j \in \N$, there exists a $k \in \N$ with $F_j(\omega, \cdot) = f_k(\omega, \cdot)$.
\end{theorem}

\begin{proof}
The function $f$ is  well defined as a function from $Gr K$ to $\C.$ For each $z\in\C^n$ such that $K^{-1}(z)\not=\emptyset,$ the function $f(\cdot,z)$ is measurable on $K^{-1}(z)$ by Proposition \ref{limsmooth}. Thus, $f$ is a generalized random function. 

Define the multifunction
\begin{equation*}
    N(\omega) = \{j : \abs{f(\omega, z) - f_k(\omega, z)} \leq \varepsilon, \text{ }\forall z \in K(\omega), \forall k \geq j\}.
\end{equation*}
This multifunction takes its values in the set of non-empty closed subsets of $\N$ since the uniform convergence guarantees non-emptiness.

Let $K'=\cup_{\omega \in \Omega} K(\omega)$. Let $O$ be an open set of $\N$ and denote by $E \subset K'$ the countable set, dense in $K(\omega)$ for all $\omega \in \Omega$. Then
\begin{align*}
&\{\omega : N(\omega) \cap O\neq \emptyset\} \\
&= \bigcup_{a \in O} \{\omega : a \in N(\omega) \} \\
&= \bigcup_{a \in O}\bigcap_{b \geq a} \{\omega : \abs{f(\omega, z) - f_b(\omega, z)} \leq \varepsilon, \text{ }\forall z\in K(\omega)\} \\
&= \bigcup_{a \in O}\bigcap_{b \geq a} \bigcap_{z \in K'} \Omega\setminus\{\omega \in K^{-1}(z): \abs{f(\omega, z) - f_b(\omega, z)} > \varepsilon \} \\
&= \bigcup_{a \in O}\bigcap_{b \geq a} \bigcap_{z \in K'} \left[\{\omega \in K^{-1}(z): \abs{f(\omega, z) - f_b(\omega, z)} \leq \varepsilon \} \cup \{\omega \not\in K^{-1}(z)\} \right]\\
&= \bigcup_{a \in O}\bigcap_{b \geq a}
\bigcap_{z \in E}\left[\{\omega \in K^{-1}(z): \abs{f(\omega, z) - f_b(\omega, z)} \leq \varepsilon \} \cup \{\omega \not\in K^{-1}(z)\} \right].
\end{align*}
We show in detail the third and fifth equalities. Fix $b \in \N$ and let
\begin{align*}
A &= \{\omega : \abs{f(\omega, z) - f_b(\omega, z)} \leq \varepsilon, \text{ }\forall z\in K(\omega)\};\\
B &= \bigcap_{z \in K'} \Omega\setminus\{\omega \in K^{-1}(z): \abs{f(\omega, z) - f_b(\omega, z)} > \varepsilon \};\\
C &= \bigcap_{z \in K'} \left[\{\omega \in K^{-1}(z): \abs{f(\omega, z) - f_b(\omega, z)} \leq \varepsilon \} \cup \{\omega \not\in K^{-1}(z)\} \right]; \\
D &= \bigcap_{z \in E}\left[\{\omega \in K^{-1}(z): \abs{f(\omega, z) - f_b(\omega, z)} \leq \varepsilon \} \cup \{\omega \not\in K^{-1}(z)\} \right].
\end{align*}
We wish to show that $A = B$ to prove the third equality.

Let $\omega' \in A$. Then, for all $z \in K(\omega')$, $\omega'$ is in $\{\omega \in K^{-1}(z): \abs{f(\omega, z) - f_b(\omega, z)} \leq \varepsilon\}$. Therefore,
\begin{align*}
&&\omega' \in \bigcap_{z \in K(\omega')} \{\omega \in K^{-1}(z): \abs{f(\omega, z) - f_b(\omega, z)}\leq \varepsilon\} &&\\
&\Rightarrow &\omega' \in \bigcap_{z \in K(\omega')} K^{-1}(z)\setminus\{\omega \in K^{-1}(z): \abs{f(\omega, z) - f_b(\omega, z)} > \varepsilon\}& \\
&\Rightarrow &\omega' \in \bigcap_{z \in K(\omega')} \Omega\setminus\{\omega \in K^{-1}(z): \abs{f(\omega, z) - f_b(\omega, z)} > \varepsilon\} &\\
&\Rightarrow &\omega' \in \bigcap_{z \in K'} \Omega\setminus\{\omega \in K^{-1}(z): \abs{f(\omega, z) - f_b(\omega, z)} > \varepsilon\} &.
\end{align*}
where the second implication follows from the fact that if $z \in K(\omega')$, then $\omega' \in K^{-1}(z)$. Thus, $A \subset B$.

Now let $\omega' \in B$. Therefore,
\begin{align*}
&&\omega' \in \bigcap_{z \in K'} \Omega\setminus\{\omega \in K^{-1}(z): \abs{f(\omega, z) - f_b(\omega, z)} > \varepsilon\} & \\
&\Rightarrow &\omega' \in \bigcap_{z \in K(\omega')} \Omega \setminus \{\omega \in K^{-1}(z): \abs{f(\omega, z) - f_b(\omega, z)} > \varepsilon\} &. 
\end{align*}\\
But for each $z\in K(\omega')$ we have $\omega'\in K^{-1}(z),$ so
\begin{align*}
&&\omega' \in \bigcap_{z \in K(\omega')} K^{-1}(z) \setminus \{\omega \in K^{-1}(z): \abs{f(\omega, z) - f_b(\omega, z)} > \varepsilon\} & \\
&\Rightarrow &\omega' \in \bigcap_{z \in K(\omega')} \{\omega \in K^{-1}(z): \abs{f(\omega, z) - f_b(\omega, z)} \leq \varepsilon\} & \\
&\Rightarrow &\omega' \in \{\omega : \abs{f(\omega, z) - f_b(\omega, z)} \leq \varepsilon, \text{ }\forall z \in K(\omega')\}.
\end{align*}
We have thus shown that $A = B$ and this justifies the third equality.

We now wish to show that $C = D$ to show the fifth equality.

Since $E \subset K'$, it is obvious that $C \subset D$. Now let $\omega' \not\in C$. Therefore, there exists a $z \in K(\omega')$ such that $\abs{f(\omega', z) - f_b(\omega', z)} > \varepsilon$. By density of $E$ in $K(\omega')$, there exists a sequence $\{z_j\} \subset K(\omega') \cap E$ such that $z_j \rightarrow z$. By continuity and convergence of the sequence, there exists $z_N \in K(\omega')\cap E$ such that $\abs{f(\omega', z_N) - f_b(\omega', z_N)} > \varepsilon$. Therefore,
\begin{eqnarray*}
\omega' \not\in \{\omega : \abs{f(\omega, z_N) - f_b(\omega, z_N)}\le\varepsilon\} & \text{and} & \omega' \not\in \{\omega \not\in K^{-1}(z_N)\}
\end{eqnarray*}
and thus
\begin{equation*}
\omega' \not\in \bigcap_{z \in E}\left[\{\omega \in K^{-1}(z): \abs{f(\omega, z) - f_b(\omega, z)} \leq \varepsilon \} \cup \{\omega \not\in K^{-1}(z)\} \right].
\end{equation*}
This shows that $D \subset C$. We have thus shown that $C = D$ and this justifies the fifth equality.

For fixed $z$, the function $\abs{f(\cdot, z) - f_b(\cdot, z)}$ as a function from $K^{-1}(z)$ to $\R$ is measurable and therefore, with the help of Lemma \ref{F sub f}, we see that 
the function $N$ is weakly measurable. It follows that $N$ admits a measurable selection $\varphi$ by Theorem \ref{krn}.

Define $F(\omega,z) = f_{\varphi(\omega)}(\omega, z)$. Let $U$ be a measurable subset of $\C$ and fix $z \in \C^n$. Then
\begin{equation*}
\{\omega \in K^{-1}(z): F(\omega, z) \in U\} = \bigcup_{j \in \N} \{\omega \in K^{-1}(z): f_j(\omega, z) \in U\} \cap \{\omega : \varphi(\omega) = j\}
\end{equation*}
and therefore $F$ is a generalized continuous random function by the measurability of $f_j$ and $\varphi$. Also, by construction, we have that
\begin{equation*}
\abs{f(\omega, z) - F(\omega, z)} = \abs{f(\omega, z) - f_{\varphi(\omega)}(\omega, z)} \leq \varepsilon
\end{equation*}
and so by taking a sequence decreasing to $0$, we can construct a sequence $\{F_j\}$ with the desired properties.
\end{proof}

In particular, we can apply the previous theorem to families of functions.
For each nonempty compact set $K\subset\C^n,$ let $A(K)$ be a family of continuous functions on $K$. Thus,
$$
	A:\mathcal K^\prime(\C^n)\rightarrow \mathcal P(C(K)).
$$
Suppose further that if $Q \subset K$ is a compact set and $f \in A(K)$, then $f|_Q \in A(Q)$. For example, $A(K)$ can be the family of holomorphic, polynomial, or harmonic functions on $K$ 
\color{black}
or the family of rational functions having no poles on $K.$ 
\color{black}
For a random compact set $K(\omega),$ we define a generalized $A(K)$-random function as  a generalized random 
function $f : \Gr K \rightarrow \C$ such that $f(\omega, \cdot) \in A(K(\omega))$ for all $\omega \in \Omega.$

\color{black}
Suppose  $K\mapsto A(K),$ for $K\in \mathcal K^\prime(\C^n),$ is as above and 
$K:\Omega\rightarrow \mathcal K^\prime(\C^n)$ is a  
random compact subset of $\C^n,$
\color{black}
Denote by $A_{[\Omega]}(K)$ the set of generalized $C(K)$-random functions $f : \Gr K \rightarrow \C,$ for which  there exists a sequence of generalized $A(K)$-random functions $\{f_j\}_{j=1}^\infty,$ 
\color{black}
such that, for each $\omega \in \Omega,$ 
\color{black}
$f_j(\omega, \cdot) \rightarrow f(\omega, \cdot)$ uniformly on $K(\omega).$
Denote by $A_{[\Omega]}^{unif}(K)$ the set of generalized $C(K)$-random functions $f : \Gr K \rightarrow \C,$ for which there exists a sequence of generalized $A(K)$-random functions $\{f_j\}_{j=1}^\infty$ such that $f_j \rightarrow f$ uniformly on $\Gr K$. 

The following theorem follows directly from the previous one.

\begin{theorem} \label{convUnif}
\color{black}
Suppose  $K\mapsto A(K),$ for $K\in \mathcal K^\prime(\C^n),$ is as above and 
$K:\Omega\rightarrow \mathcal K^\prime(\C^n)$ is a uniformly separable 
random compact subset of $\C^n,$
\color{black} 
then $A_{[\Omega]}^{unif}(K) = A_{[\Omega]}(K)$.
\end{theorem}

\color{black}
Suppose  $K\mapsto A(K),$ for $K\in \mathcal K^\prime(\C^n),$ is as above. 
Denote by $A_{\Omega}(K)$ the set of  $C(K)$-random functions $f : \Omega\times K \rightarrow \C,$ for which  there exists a sequence of $A(K)$-random functions $\{f_j\}_{j=1}^\infty,$ 
such that, for each $\omega \in \Omega,$ 
$f_j(\omega, \cdot) \rightarrow f(\omega, \cdot)$ uniformly on $K.$
Denote by $A_{\Omega}^{unif}(K)$ the set of  $C(K)$-random functions $f : \Omega\times K \rightarrow \C,$ for which there exists a sequence of  $A(K)$-random functions $\{f_j\}_{j=1}^\infty$ such that $f_j \rightarrow f$ uniformly on $\Omega\times K.$

\begin{corollary}\label{unif}
If $K\mapsto A(K),$ for $K\in \mathcal K^\prime(\C^n),$ is as above, then 
$A_{\Omega}^{unif}(K) = A_{\Omega}(K)$.
\end{corollary}

\color{black}

We shall now, following rather closely the method of \cite{AB1984}, prove an "almost everywhere" version of a random Oka-Weil theorem. We first need the following lemmas :

\begin{lemma}\label{P closed}
	The set $P(k,m)$ of polynomials in $\C^n$ of degree at most $k$ and coefficients bounded by $m$ (in norm) is closed in every $C(K,\C)$ for every non-empty compact set $K\subseteq \C^n$.
\end{lemma}

\begin{proof}
	Fix a non-empty compact set $K$. Let there be a sequence $p_i$ of polynomials in $P(k,m)$ converging uniformly to a function $f\in C(K,\C)$. We wish to show $f$ is in the restriction of $P(k,m)$ over $K$.
	
	We denote $d_k$ the total number of possible terms in such polynomials. We in fact have that
	\begin{align*}
		d_k=\sum_{\ell=0}^k\binom{n+\ell-1}{n-1}.
	\end{align*}
	
	We denote $a_{i,\alpha}$, where $\alpha\in \N^n$ the coefficient of $z^\alpha$. We may then form the $d_k$-tuplets 
 of coefficients (by choosing an ordering) : $a_i\in \C^{d_k}$ for $p_i$. To simplify our work, we may endow $\C^{d_k}$ with the max norm :
$$\|a_i\|_\infty=\max_\alpha |a_{i,\alpha}|.$$
	
	The $a_{i,\alpha}$ are bounded in norm by $m$. Thus, the $a_i$ are all in the ball $\overline{B_{m}}(0)\subset\C^{d_k} $, which is a compact set, and thus  the sequence $\{a_i\}_{i=0}^\infty$ has a subsequence $\{a_i'\}_{i=0}^\infty$ which converges to a point $b\in\C^{d_k}.$ We denote the associated subsequence of polynomials as $\{p_i'\}_{i=0}^\infty$. 

Since $p_i\rightarrow f,$ the same is true for the subsequence $p_i'.$ But $p_i'$ converges on all compact sets to the polynomial $p(z)=\sum_{|\alpha|\leq k} b_\alpha z^\alpha.$ Thus, $f=p|_K.$ 
\end{proof}

\begin{lemma}
	The set of polynomials $\mathcal P(\C^n)$ over $\C^n$ is a Suslin subset of $C(K,\C)$ for every non-empty compact set $K\subseteq \C^n$.
\end{lemma}
\begin{proof}
	Firstly, we have that
	\begin{align*}
		\mathcal P(\C^n)=\cup_{m=0}^\infty\cup_{k=0}^\infty P(k,m).
	\end{align*}
	If we fix a compact set $K\subset \C^n$, this is the countable union of sets, which by Lemma \ref{P closed} are closed relative to $C(K,\C)$. Thus, $\mathcal P(\C^n)|_K$ is a Borel subset of $C(K,\C)$.
	
	We have already shown in the proof of Proposition \ref{random element} that $C(K,\C)$ is separable. It is also complete, as a Cauchy sequence with the sup norm converges uniformly, and thus converges toward a continuous function. Thus, $C(K,\C)$ is a Polish space.
	
	We know that a Borel subset of a Polish space is Suslin.
\end{proof}

\begin{theorem} \label{OkaWeil}
	Let $(\Omega,\mathcal A, \mu)$ be a $\sigma$-finite measure space. Let $K$ be a random compact set that takes at most a countable number of different values. Let $\varepsilon$ be a positive measurable function defined on $\Omega.$ 
	
	Let $K$ be polynomially convex : $\forall \omega, K(\omega)=\widehat K(\omega)$. Then if $f$ is a generalized random function such that $f(\omega,\cdot)$ is holomorphic in a neighborhood of $K(\omega)$, for each $\omega$, then there exists a generalized random polynomial $p$ such that$$\|p(\omega,\cdot)-f(\omega,\cdot)\|_{K(\omega)}<\varepsilon(\omega)$$
	for all $\omega$ except a set $L\subset \Omega$ such that $\mu(L)=0$.
\end{theorem}
\begin{proof}
	We denote $\{K_j\}_{j=0}^\infty$ the values taken by $K$. We also denote 
	\begin{equation*}
	\Omega_j=\{\omega \in \Omega : K(\omega)=K_j\}=K^{-1}(K_j).
	\end{equation*}
As this is the pre-image of a closed set (the single compact set), each $\Omega_j$ is measurable.
	
	We will define as measurable sets of $\Omega_j$ the following $\sigma$-algebra : $\mathcal A_j=\{P\cap \Omega_j : P\in \mathcal A\}$. But $\Omega_j$ is measurable in $\Omega$, and thus this definition is equivalent to $\mathcal A_j=\{P\subseteq \Omega_j : P\in \mathcal A\}$. Most importantly, every measurable set of $\Omega_j$ is measurable in $\Omega$.
	
	Let $f_j$ be the restriction of $f$ to $\Omega_j\times K_j.$ We claim this is a random function. Fix $z\in K_j$. We know that $f(\cdot, z)$ is a measurable mapping between $K^{-1}(z)$ and $\C$ and, for an open set $O$, $$(f_j(\cdot,z))^{-1}(O)=\Omega_j\cap (f(\cdot,z))^{-1}(O).$$ This proves that $f(\cdot,z)$ is a measurable function over $\Omega_j$, and thus $f$ is a random function.
	
	By Proposition \ref{random element}, we know that the mapping $\omega\mapsto f(\omega,\cdot)$ between $\Omega_j$ and $C(K_j,\C)$ is measurable.
	
	We know that $f_j(\omega,\cdot)$ is holomorphic in a neighbourhood of $K_j$. But $K_j$ is, by hypothesis, polynomially convex. Thus, we have by the Oka-Weil Approximation Theorem that the multivalued mapping
	\begin{align*}
		\psi_{j}(\omega)=\{q\in \mathcal P(\C^n) : \|f_j(\omega,\cdot)-q(\cdot)\|_{K_j}<\varepsilon(\omega)\}
	\end{align*}
	is never empty.
	
	By Theorem 1 of \cite{AB1984}, and since $\mathcal P(\C^n)$ is a Suslin subset of $C(K_j,\C)$, this means there exists a measurable selection $p_j : \Omega_j \rightarrow \mathcal P(\C^n)$ which approximates $f$ almost everywhere.
	
	We now define $p : \Omega\rightarrow\mathcal P(\C^n)$ by setting $p|_{\Omega_j}=p_j$. Clearly, this function approximates $f$ almost everywhere. Indeed, if we denote $L_j$ the set of $\omega\in \Omega_j$ for which $p$ does not approximate $f$, we have that $L=\cup_{j=0}^\infty L_j$ and thus $\mu(L)=\sum_{i=0}^\infty \mu(L_j)=0$.
	
	It is only left to show that this is a generalized random function. By Proposition \ref{random element}, applied to the uniform algebra $A_j=C(K_j,\C),$ we have that for fixed $z\in K_j$, the mapping $\omega \mapsto p_j(\omega,z)$ is measurable. Thus, we can now say that, for a given open set $O\subset \C$ and $z \in \C^n$, the set
	\begin{align*}
		(p(\cdot,z))^{-1}(O)=\bigcup_{\{j\in \N: z\in K_j\}}(p_j(\cdot,z))^{-1}(O)
	\end{align*}
is a measurable set as a countable union of measurable sets. 

Since this is a subset of the measurable set 
$$
	K^{-1}(z)=\bigcup_{j : z\in K_j}\Omega_j,
$$
the function $p$ is measurable on $K^{-1}(z)$ and hence $p$ is a generalized random function.
\end{proof}

We might be interested to see what happens when trying to remove the "almost everywhere" or by reducing the hypothesis of the compact set : could it be true if it were only uniformly separable, or even only random?


\section{Annex}\label{annex}

Recall that we here use the Borel $\sigma$-algebra on $\R^n$ (and $\C^n$  viewed as $\R^{2n}$) to define measurability. To facilitate the reading of this paper, we collect here some well known basic facts which have been used  throughout the paper. 

\begin{proposition} \label{component}
Let $(\Omega, \mathcal A)$ be a measurable space; let $f_i:\Omega\rightarrow \mathbb R$ for each $1\leq i \leq n$ and put $f=(f_1,\ldots,f_n).$ 
Then  $f : \Omega\rightarrow \mathbb R^n$ is measurable if and only if each $f_i$ is measurable.

Let $f:\Omega\rightarrow\C^n.$ By the canonical identification $\C^n=\R^{2n}$ and writing 
$f=u+iv$, where $u,v:\Omega\rightarrow\R^n,$ it follows that $f$ is measurable if and only if $u$ and $v$ are measurable. 
\end{proposition}

\begin{proposition} \label{cont}
If $X$ be a topological space and let $f : X\rightarrow \mathbb R^n$. Let $f=(f_1,\dots,f_n)$ where $f_i:X\rightarrow \mathbb R$ for each $1\leq i \leq n$. Then $f$ is continuous if and only if each $f_i$ is continuous.

By canonically identifying $\C^n$ with $\R^{2n}$, the analogous statement follows for $\C^n$.

\end{proposition}
This can be proved with essentially the same proof as the preceding lemma, since an open set in  
 $\R^{2n}$ is a countable union of open $n$-cubes.

\begin{proposition}[\cite{Ma}]\label{limsmooth}
Let $(\Omega, \mathcal A)$ be a measurable space and $Y$ a metric space. Suppose $g_n : \Omega\rightarrow Y$ are measurable functions and $g : \Omega\rightarrow Y$ such that $\lim_{n\rightarrow \infty} g_n = f$ pointwise. Then $g$ is measurable.
\end{proposition}

\begin{proposition}\label{sum}
Let $(\Omega, \mathcal A)$ be a measurable space. Suppose $f : \Omega\rightarrow \mathbb R$ and $g : \Omega\rightarrow \mathbb R$ are measurable functions. Then $f+g$ and $f \cdot g$ are measurable.
\end{proposition}

By applying Proposition \ref{component} repeatedly, one easily finds analogues for Propositions \ref{limsmooth} and \ref{sum} in $\R^n$ and $\C^n$.

\begin{proposition} \label{randomInt}
Let $(\Omega,\mathcal A)$ be a measurable space and $Q$ a hypercube in $\R^n.$ Suppose $f : \Omega \times Q \rightarrow\C$ is such that $f(\cdot, x)$ is measurable for all $x$ and $f(\omega, \cdot)$ is Riemann integrable on  $Q$ for all $\omega \in \Omega$. Then the function $F(\omega) = \int_Q f(\omega,x) \diff x$ is measurable.
\end{proposition}

\begin{proof}
Since $f(\omega,\cdot)$ is Riemann integrable for each $\omega,$ we have pointwise convergence of the Riemann sums to $F$. By Proposition \ref{limsmooth}, we conclude that $F$ is measurable.
\end{proof}

\begin{theorem}[Kuratowski, Ryll-Nardzewski \cite{KRN}]\label{krn}
    Let  $X$ be a complete separable metric  space, $\mathcal B$ the Borel $\sigma$-algebra of $X$, $(\Omega, \mathcal A)$ a measurable space and $\phi$ a multifunction on $\Omega$ taking values in the set of nonempty closed subsets of $X$. Suppose that $\phi$ is $\mathcal A$-weakly measurable, that is, for every open set $U$ of $X$, we have
    \begin{equation*}
        \{\omega : \phi(\omega) \cap U \neq \emptyset\} \in \mathcal A\}.
    \end{equation*}
    Then $\phi$ admits a selection that is measurable.
\end{theorem}

If $U$ is an open subset of $\R^N$, and $(\Omega,\mathcal A)$ is a measurable space, a random function $f : \Omega \times U \rightarrow \R$ is a random smooth function on $U$ if $f(\omega, \cdot)$ is $C^\infty$ for each $\omega \in \Omega$. For each multi-index $\alpha=(\alpha_1,\ldots,\alpha_N), \, \alpha_j\in\N,$ we denote by $f^{(\alpha)}$ the function $f^{(\alpha)}:\Omega\times U\rightarrow \R,$ which for every fixed $\omega\in\Omega$ is defined as
\begin{equation*}
    f(\omega,\cdot)^{(\alpha)} = \frac{\partial^{\alpha_1 + \ldots + \alpha_N} f}{\partial x_1^{\alpha_1}\dots \partial x_N^{\alpha_N}}(\omega,\cdot).
\end{equation*}

\begin{lemma} \label{smoothlemma}
If $f$ is a random smooth funtion, then, for each muli-index $\alpha,$ the function $f^{(\alpha)}$ is also a random smooth function.
\end{lemma}

\begin{proof}
    It is sufficient to show the proposition for $\alpha = (1,0,\ldots,0)$ since the proof of the lemma follows from symmetry and induction. Consider the differential quotient
\begin{equation*}
    g_k(\omega,x) = \frac{f(\omega, x_1 + \frac{1}{k}, x_2, \ldots , x_N) - f(\omega, x_1, \ldots, x_N)}{\frac{1}{k}}.
\end{equation*}
For fixed $k$ (big enough in order for $g_k$ to be well-defined), the function $g_k(\cdot,x)$ is measurable, and by differentiablity of $f$, we have
\begin{equation*}
    \lim_{k\rightarrow\infty} g_k(\omega,x) = \pardiff{f}{x_1}(\omega,x) = f(\omega,x)^{(\alpha)}.
\end{equation*}
By Proposition \ref{limsmooth}, we have that $f(\cdot,x)^{(\alpha)}$ is the limit of measurable functions and is therefore measurable. Similarly, since $f(\omega,\cdot)$ is smooth on $U$, $f(\omega,\cdot)^{(\alpha)}$ is also smooth on $U$. Thus $f^{(\alpha)}$ is a random smooth function.
\end{proof}


\end{document}